\documentclass[11pt]{amsart}
\usepackage[margin=1.35in]{geometry}
\usepackage{amscd,amsmath,amsxtra,amsthm,amssymb,stmaryrd,xr,mathrsfs,mathtools,enumerate,commath, comment}
\usepackage{stmaryrd}
\usepackage{xcolor}
\usepackage{commath}
\usepackage{comment}
\usepackage{tikz-cd}
\usepackage{float}
\usepackage{longtable} % for 'longtable' environment
\usepackage{pdflscape} % for 'landscape' environment
\usepackage{booktabs}
\usepackage{hyperref}
\usepackage[capitalize,nameinlink,noabbrev]{cleveref}

\usepackage{multirow}
\usepackage[utf8]{inputenc}
\definecolor{vegasgold}{rgb}{0.77, 0.7, 0.35}
\definecolor{darkgoldenrod}{rgb}{0.72, 0.53, 0.04}
\definecolor{gold(metallic)}{rgb}{0.83, 0.69, 0.22}
\hypersetup{
 colorlinks=true,
 linkcolor=darkgoldenrod,
 filecolor=brown,      
 urlcolor=gold(metallic),
 citecolor=darkgoldenrod,
 pdftitle={Growth Formula for Ramified Iwasawa Towers of Graphs},
 }

\usepackage[all,cmtip]{xy}

\usetikzlibrary{shapes.geometric}
\usetikzlibrary{decorations.markings}
\usetikzlibrary{shapes.geometric}
\tikzset{every loop/.style={min distance=10mm,looseness=10}}

\DeclareFontFamily{U}{wncy}{}
\DeclareFontShape{U}{wncy}{m}{n}{<->wncyr10}{}
\DeclareSymbolFont{mcy}{U}{wncy}{m}{n}
\DeclareMathSymbol{\Sh}{\mathord}{mcy}{"58}
\usepackage[T2A,T1]{fontenc}
\usepackage[OT2,T1]{fontenc}

\newtheorem{theorem}{Theorem}[section]
\newtheorem{lemma}[theorem]{Lemma}

\newtheorem{proposition}[theorem]{Proposition}
\newtheorem{corollary}[theorem]{Corollary}
\newtheorem{definition}[theorem]{Definition}
\newtheorem{assumption}[theorem]{Assumption}

\numberwithin{equation}{section}

\newtheorem{lthm}{Theorem} % theorems with letters (for intro)

\theoremstyle{remark}
\newtheorem{remark}[theorem]{Remark}

\begin{document}
\title[Iwasawa theory for branched $\mathbb{Z}_{p}$-towers of finite graphs]{Iwasawa theory for branched $\mathbb{Z}_{p}$-towers of finite graphs}

\author[R.~Gambheera]{Rusiru Gambheera}
\address{Rusiru Gambheera\newline Department of Mathematics, University of California Santa Barbara, CA 93106-3080, USA}
\email{rusiru@ucsb.edu}

\author[D.~Valli\`{e}res]{Daniel Valli\`{e}res}
\address{Daniel Valli\`{e}res\newline Mathematics and Statistics Department, California State University, Chico, CA 95929, USA}
\email{dvallieres@csuchico.edu}

\begin{abstract}
We initiate the study of Iwasawa theory for branched $\mathbb{Z}_{p}$-towers of finite connected graphs.  These towers are more general than what have been studied so far, since the morphisms of graphs involved are branched covers, a particular kind of harmonic morphisms of graphs.  We prove an analogue of Iwasawa's asymptotic class number formula for the $p$-part of the number of spanning trees in this setting.  Moreover, we find an explicit generator for the characteristic ideal of the torsion Iwasawa module governing the growth of the $p$-part of the number of spanning trees in such towers. 
\end{abstract}

\date{\today} 

\maketitle
\tableofcontents

\section{Introduction} \label{Introduction}
Let $p$ be a rational prime and let 
$$K = K_{0} \subseteq K_{1} \subseteq K_{2} \subseteq \ldots \subseteq K_{n} \subseteq \ldots$$
be a tower of number fields for which $K_{n}/K$ is Galois with Galois group isomorphic to $\mathbb{Z}/p^{n}\mathbb{Z}$.  Iwasawa's famous asymptotic class number formula (\cite{Iwasawa:1959} and \cite{Iwasawa:1973}) shows the existence of $\mu,\lambda,n_{0} \in \mathbb{Z}_{\ge 0}$ and $\nu \in \mathbb{Z}$ such that when $n \ge n_{0}$, one has
$${\rm ord}_{p}(h(K_{n})) = \mu p^{n} + \lambda n + \nu,$$
where $h(K_{n})$ denotes the class number of $K_{n}$, and ${\rm ord}_{p}$ denotes the usual $p$-adic valuation on the field of rational numbers.

Motivated by the analogy between number theory and graph theory, one can study the variation of the $p$-part of the number of spanning trees as one goes up a $\mathbb{Z}_{p}$-tower of graphs
\begin{equation} \label{tower}
X = X_{0} \leftarrow X_{1} \leftarrow X_{2} \leftarrow \ldots \leftarrow X_{n} \leftarrow \ldots
\end{equation}
Such a $\mathbb{Z}_{p}$-tower consists of a sequence of covering maps $X_{n+1} \rightarrow X_{n}$ between finite connected graphs for which the composition $X_{n} \rightarrow X_{n-1} \rightarrow \ldots \rightarrow X_{0} = X$ is Galois with Galois group isomorphic to $\mathbb{Z}/p^{n}\mathbb{Z}$.  It turns out that in perfect analogy with Iwasawa's asymptotic class number formula, there exist $\mu,\lambda,n_{0} \in \mathbb{Z}_{\ge 0}$ and $\nu \in \mathbb{Z}$ such that
\begin{equation} \label{analogue_Iw}
{\rm ord}_{p}(\kappa(X_{n})) = \mu p^{n} + \lambda n + \nu,
\end{equation}
provided $n \ge n_{0}$, where now $\kappa(X_{n})$ is the number of spanning trees of $X_{n}$.  

A proof of this result is given in increasing generality in \cite{Vallieres:2021,mcgownvallieresII,mcgownvallieresIII}.  The approach therein is based on Ihara zeta and $L$-functions and their special value at $u=1$.  The Ihara zeta function $Z_{X}(u)$ of a finite graph is the reciprocal of a polynomial.  It is of the form $Z_{X}(u)^{-1} = (1-u^{2})^{-\chi(X)} h_{X}(u)$, where $\chi(X)$ denotes the Euler characteristic of $X$, and where $h_{X}(u)$ is some polynomial with integer coefficients given explicitly as the determinant of an operator via Ihara's determinant formula.  A result of Hashimoto (\cite{Hashimoto:1990}) gives
$$h_{X}^{'}(1) = -2 \chi(X) \kappa(X),$$
when $X$ is a finite connected graph.  Therefore, the special value at $u=1$ of these functions encodes much information about the invariant $\kappa(X)$.  These facts can be used to construct a power series $f(T) \in \mathbb{Z}_{p}\llbracket T \rrbracket$, and combined with the Artin formalism satisfied by the Ihara $L$-functions, one can get the Iwasawa invariants $\mu$ and $\lambda$ from the power series $f(T)$ thereby obtaining (\ref{analogue_Iw}).  

On the other hand, another proof of (\ref{analogue_Iw}) has been obtained by Gonet in \cite{Gonet:2022}.  Her approach is algebraic and makes use of the structure theorem for Iwasawa modules.  To every finite connected graph is associated a finite abelian group ${\rm Pic}^{0}(X)$ which is called the Picard group of degree zero of $X$ (also known as the Jacobian, the sandpile, or still the critical group of $X$).  It follows from a classical theorem of Kirchhoff in graph theory that its cardinality is precisely the invariant $\kappa(X)$.  Any cover of finite connected graphs $f:Y \rightarrow X$ induces a surjective group morphism $f_{*}:{\rm Pic}^{0}(Y) \rightarrow {\rm Pic}^{0}(X)$.  Thus from a $\mathbb{Z}_{p}$-tower of graphs such as in (\ref{tower}) above, one obtains a compatible system of maps ${\rm Pic}^{0}(X_{n+1})[p^{\infty}] \rightarrow {\rm Pic}^{0}(X_{n})[p^{\infty}]$, where we write $A[p^{\infty}]$ for the Sylow $p$-subgroup of a finite abelian group $A$.  Since each ${\rm Pic}^{0}(X_{n})[p^{\infty}]$ is a $\mathbb{Z}_{p}[\mathbb{Z}/p^{n}\mathbb{Z}]$-module, the inverse limit
\begin{equation}
{\rm Pic}^{0}_{\Lambda} = \varprojlim_{n \ge 0}{\rm Pic}^{0}(X_{n})[p^{\infty}]
\end{equation}
becomes a module over the Iwasawa algebra $\Lambda$ of $\mathbb{Z}_{p}$, and one can use the structure theorem for finitely generated $\Lambda$-modules up to pseudo-isomorphisms to obtain a module theoretical proof of (\ref{analogue_Iw}).  

The two approaches are intimately related to one another.  Using the usual non-canonical isomorphism $\Lambda \stackrel{\simeq}{\longrightarrow} \mathbb{Z}_{p}\llbracket T \rrbracket$ given by $\gamma \mapsto 1+T$, every $\Lambda$-module can be viewed as a $\mathbb{Z}_{p}\llbracket T \rrbracket$-module, and Kleine and M\"{u}ller proved (\cite[Theorem 5.2 and Remark 5.3]{Kleine/Muller:2023}) that 
\begin{equation} \label{mm}
{\rm char}_{\mathbb{Z}_{p}\llbracket T \rrbracket} ({\rm Pic}^{0}_{\Lambda}) \cdot (T) = (f(T)),
\end{equation}
where ${\rm char}_{\mathbb{Z}_{p}\llbracket T \rrbracket}(M)$ denotes the characteristic ideal of a finitely generated $\mathbb{Z}_{p}\llbracket T \rrbracket$-module $M$ and $f(T)$ is the power series constructed in \cite{Vallieres:2021,mcgownvallieresII,mcgownvallieresIII}.  For the module theoretical approach to both (\ref{analogue_Iw}) and (\ref{mm}), in addition to \cite{Gonet:2022} and \cite{Kleine/Muller:2023}, see also \cite{Kataoka:2024}. 
\begin{remark}
The equality (\ref{mm}) is not quite true, since the non-canonical isomorphism $\gamma \mapsto 1-T$ was used in \cite{mcgownvallieresIII} instead of $\gamma \mapsto 1 + T$.  So it is true when replacing $T$ with $-T$ on the right hand side of (\ref{mm}).  
\end{remark}

In the analogy between graphs and compact Riemann surfaces, covering maps of graphs correspond to unramified covers of compact Riemann surfaces.  On the other hand, branched (or ramified) covers of compact Riemann surfaces correspond to harmonic morphisms, a broader class of morphisms of graphs than covering maps.  Harmonic morphisms were introduced in \cite{Urakawa:2000} and studied further by many authors including for instance in \cite{Baker/Norine:2009} and \cite{Corry:2012}.  The goal of this paper is to study more general $\mathbb{Z}_{p}$-towers of graphs than the ones that have been considered so far in the literature.  We will consider towers 
\begin{equation*}
X = X_{0} \leftarrow X_{1} \leftarrow X_{2} \leftarrow \ldots \leftarrow X_{n} \leftarrow \ldots
\end{equation*}
such as (\ref{tower}) above, but where the maps $X_{n+1} \rightarrow X_{n}$ will be \emph{branched} (or \emph{ramified}) covers of finite graphs rather than usual (unramified) covering maps.  Branched covers of graphs form a class of graph morphisms that lies in between (unramified) covering maps and harmonic morphisms.  We explain in \cref{towers} how to construct such branched $\mathbb{Z}_{p}$-towers of finite connected graphs by modifying the notion of voltage assignment accordingly.  In this paper, we use the module theoretical approach.  There should be a connection with Ihara zeta and $L$-functions as well in this situation, as studied for instance in \cite{Bass:1992} and \cite{Zakharov:2021}, and we would like to look at this point of view in a separate work.  The reason why we adopt the module theoretical approach here is that given a branched cover of finite connected graphs $Y \rightarrow X$, it induces a natural surjective group morphism ${\rm Pic}^{0}(Y) \rightarrow {\rm Pic}^{0}(X)$ just as in the unramified situation.  Therefore, starting with a branched $\mathbb{Z}_{p}$-tower of graphs, we obtain again a compatible system of morphisms ${\rm Pic}^{0}(X_{n+1})[p^{\infty}] \rightarrow {\rm Pic}^{0}(X_{n})[p^{\infty}]$, and we can consider the $\Lambda$-module
$${\rm Pic}_{\Lambda}^{0} = \varprojlim_{n \ge 0} {\rm Pic}^{0}(X_{n})[p^{\infty}]. $$  
The study of this Iwasawa module leads us to our first main result.
\begin{lthm}[\cref{analogue_iwasawa}] \label{main1}
Let 
$$X = X_{0} \leftarrow X_{1} \leftarrow X_{2} \leftarrow \ldots \leftarrow X_{n} \leftarrow \ldots$$
be a branched $\mathbb{Z}_{p}$-tower of finite connected graphs arising from a voltage assignment as explained in \cref{towers}.  Then ${\rm Pic}_{\Lambda}^{0}$ is a finitely generated torsion $\Lambda$-module.  Moreover, letting $\mu = \mu({\rm Pic}_{\Lambda}^{0})$ and $\lambda = \lambda({\rm Pic}_{\Lambda}^{0})$, there exist $n_{0}\in \mathbb{Z}_{\ge 0}$ and $\nu \in \mathbb{Z}$ such that
$${\rm ord}_{p}(\kappa(X_{n})) = \mu p^{n} + \lambda n + \nu, $$
when $n \ge n_{0}$. 
\end{lthm}

We then move on to study the characteristic ideal of ${\rm Pic}_{\Lambda}^{0}$, and we prove the following result which we now state in an imprecise form.  See \cref{main_conj} and \cref{main_conj_power_sr} for the precise formulation.
\begin{lthm}[\cref{main_conj} and \cref{main_conj_power_sr}] \label{main3}
Let 
$$X = X_{0} \leftarrow X_{1} \leftarrow X_{2} \leftarrow \ldots \leftarrow X_{n} \leftarrow \ldots$$
be a branched $\mathbb{Z}_{p}$-tower of finite connected graphs arising from a voltage assignment as explained in \cref{towers}.  We define an operator $\Delta$ on a free $\mathbb{Z}_{p}\llbracket T \rrbracket$-module $M$ of finite rank for which
$${\rm char}_{\mathbb{Z}_{p}\llbracket T \rrbracket}({\rm Pic_{\Lambda}^{0}}) \cdot (T) = ({\rm det}(\Delta)). $$
Moreover, ${\rm det}(\Delta) \in \mathbb{Z}_{p}\llbracket T \rrbracket$ can be explicitly calculated in terms of the voltage assignment.
\end{lthm}

The paper is organized as follows.  In \cref{Preliminaries}, we gather together some basic facts that will be used throughout the paper.  We remind the reader about Serre's formalism for graphs in \cref{graph_theory}, about the Picard group of degree zero in \cref{picard1}, and about group acting on graphs in \cref{group}.  When a group acts on a graph, the theory becomes an equivariant one, and we explain this for the Picard group in \cref{equivariant}.  Ultimately, we are interested in the $p$-part of the number of spanning trees, and thus we tensor everything with $\mathbb{Z}_{p}$ over $\mathbb{Z}$ in \cref{tensoring_z_p}.  We remind the reader about the Iwasawa algebra and Iwasawa modules in \cref{Iwasawa_algebra}.  In \cref{branched_cov}, we introduce the notion of branched cover of graphs $Y \rightarrow X$, and we explain a useful induced map on the Picard groups ${\rm Pic}^{0}(Y) \rightarrow {\rm Pic}^{0}(X)$.  In \cref{voltage}, we introduce the branched $\mathbb{Z}_{p}$-towers of graphs that we will be studying in this paper.  We construct branched covers explicitly using voltage assignments in \cref{basic} and we explain how this construction behaves functorially in \cref{functoriality}.  This allows us to construct branched $\mathbb{Z}_{p}$-towers of finite graphs in \cref{towers}.  To every branched $\mathbb{Z}_{p}$-tower of finite graphs is associated an unramified $\mathbb{Z}_{p}$-tower.  This is explained in \cref{immersions}, \cref{connectedness}, and \cref{useful_lemma}.  In \cref{iwasawa_theory}, we study the Iwasawa module ${\rm Pic}_{\Lambda}^{0}$ associated to a branched $\mathbb{Z}_{p}$-tower of finite connected graphs.  We prove the analogue of Iwasawa's asymptotic class number formula in \cref{asymptotic1}, and we find a generator for the characteristic ideal of ${\rm Pic}_{\Lambda}^{0}$ in \cref{main_conj1}.  This allows us to end this paper with a few numerical examples in \cref{examples}.

\subsection*{Acknowledgments}

DV would like to thank the Max Planck Institute for Mathematics in Bonn for its hospitality and financial support, where this project started.  DV would also like to thank Cristian Popescu and the department of mathematics at the University of California, San Diego, for hosting a sabbatical visit in 2023.  It is a pleasure for both authors to thank Kwun Chung, Nandagopal Ramachandran, Xie Wu, Wei Yin, and Cristian Popescu for several stimulating discussions.

%\section{Preliminaries}
\section{Preliminaries} \label{Preliminaries}
%\subsection{Graph theory}
\subsection{Graph theory} \label{graph_theory}
Throughout this paper, we use Serre's formalism for graphs (see \cite{Serre:1977} and \cite{Sunada:2013}).  Thus, a graph $X$ consists of a vertex set $V_{X}$ and a set of directed edges $\mathbf{E}_{X}$ equipped with an incidence function ${\rm inc}:\mathbf{E}_{X} \rightarrow V_{X} \times V_{X}$ given by $e \mapsto {\rm inc}(e) = (o(e),t(e))$ and an inversion function ${\rm inv}: \mathbf{E}_{X} \rightarrow \mathbf{E}_{X}$ also denoted by $e \mapsto \bar{e}$ satisfying the following conditions:
\begin{enumerate}
\item $\bar{e} \neq e$,
\item $\bar{\bar{e}} = e$,
\item $o(\bar{e}) = t(e)$ and $t(\bar{e}) = o(e)$,
\end{enumerate}
for all $e \in \mathbf{E}_{X}$.  The vertex $o(e)$ is called the origin and the vertex $t(e)$ the terminus of the directed edge $e$.  Given $v \in V_{X}$, we let
$$\mathbf{E}_{X,v} = \{e \in \mathbf{E}_{X}\, : \, o(e) = v \}. $$
We will also write $\mathbf{E}_{X,v}^{o}$ instead of $\mathbf{E}_{X,v}$ at times.  Similarly to $\mathbf{E}_{X,v}^{o} = \mathbf{E}_{X,v}$, one can define 
$$\mathbf{E}_{X,v}^{t} = \{e \in \mathbf{E}_{X} : t(e) = v \},$$
and the inversion map induces bijections
$${\rm inv}: \mathbf{E}_{X,v}^{o} \rightarrow \mathbf{E}_{X,v}^{t}$$
for all $v \in V_{X}$.  The set of undirected edges is obtained by identifying $e$ with $\bar{e}$ and will be denoted by $E_{X}$.  An orientation $S$ for $X$ consists of the image of a section for the natural map $\mathbf{E}_{X} \rightarrow E_{X}$.  In other words, an orientation is obtained by choosing a direction for each undirected edge.  A graph is finite if both $V_{X}$ and $\mathbf{E}_{X}$ are finite sets, and locally finite if $\mathbf{E}_{X,v}$ is finite for all $v \in V_{X}$.  In this case, we let
$${\rm val}_{X}(v) = |\mathbf{E}_{X,v}| = |\mathbf{E}_{X,v}^{t}|, $$
and this quantity is called the valency (or degree) of the vertex $v$.  A path in $X$ is a sequence of directed edges $c = e_{1} \cdot \ldots \cdot e_{n}$ for which one has $t(e_{i}) = o(e_{i+1})$ for all $i = 1,\ldots,n-1$.  The vertices $o(e_{1})$ and $t(e_{n})$ are called the origin and the terminus of the path $c$ and will be denoted by $o(c)$ and $t(c)$, respectively.  The graph $X$ is called connected if given any two distinct vertices $v_{1}$ and $v_{2}$ of $X$, there exists a path $c$ in $X$ satisfying $o(c) = v_{1}$ and $t(c) = v_{2}$.
\begin{definition} \label{morphism}
Let $X$ and $Y$ be graphs.  A morphism $f:X \rightarrow Y$ of graphs consists of two functions $f_{V}:V_{X} \rightarrow V_{Y}$ and $f_{E}:\mathbf{E}_{X} \rightarrow \mathbf{E}_{Y}$ satisfying
\begin{enumerate}
\item $f_{V}(o(e)) = o(f_{E}(e))$,
\item $f_{V}(t(e)) = t(f_{E}(e))$,
\item $\overline{f_{E}(e)} = f_{E}(\bar{e})$,
\end{enumerate}
for all $e \in \mathbf{E}_{X}$.  
\end{definition}
We will usually denote both $f_{V}$ and $f_{E}$ simply by $f$.  Note that given a morphism of graphs $f:X \rightarrow Y$ and any vertex $v \in V_{X}$, the restriction $f|_{\mathbf{E}_{X,v}}$ induces a function
$$f|_{\mathbf{E}_{X,v}}: \mathbf{E}_{X,v} \rightarrow  \mathbf{E}_{Y,f(v)}.$$
\begin{remark}
There are more general notions of morphisms of graphs in the literature than Definition \ref{morphism} for which one allows edges to be mapped to vertices.  See for instance \cite[Definition 2.1]{Corry:2012} in the case where the graphs are assumed to have no loops.  In this paper, we restrict ourselves to Definition \ref{morphism}.
\end{remark}

%\subsection{The Picard group of degree zero}
\subsection{The Picard group of degree zero} \label{picard1}
Let $X = (V_{X},\mathbf{E}_{X})$ be a graph.  We define ${\rm Div}(X)$ to be the free abelian group on $V_{X}$, and we have a natural surjective group morphism $s:{\rm Div}(X) \twoheadrightarrow \mathbb{Z}$ defined on vertices via $v \mapsto s(v) = 1$.  We let ${\rm Div}^{0}(X)$ be the kernel of $s$ so that we have a short exact sequence of abelian groups
\begin{equation} \label{ses_basic}
0 \rightarrow {\rm Div}^{0}(X) \rightarrow {\rm Div}(X) \stackrel{s}{\rightarrow} \mathbb{Z} \rightarrow 0.
\end{equation}

If $X$ is locally finite, then we define a few operators on ${\rm Div}(X)$ as follows.  The valency (or degree) operator $\mathcal{D}_{X}$ is defined on vertices via
$$v \mapsto \mathcal{D}_{X}(v) = {\rm val}_{X}(v)v, $$
whereas the adjacency operator $\mathcal{A}_{X}$ is defined on vertices via
$$v \mapsto \mathcal{A}_{X}(v) = \sum_{v \in \mathbf{E}_{X,v}}t(e).$$
The Laplacian operator $\mathcal{L}_{X}:{\rm Div}(X) \rightarrow {\rm Div}(X)$ is then defined to be 
$$\mathcal{L}_{X} = \mathcal{D}_{X} - \mathcal{A}_{X}.$$
Moreover, we define ${\rm Pr}(X) = {\rm Im}(\mathcal{L}_{X})$.  It is simple to check from the definitions that ${\rm Pr}(X)$ is a subgroup of ${\rm Div}^{0}(X)$.  In what follows, we let
$${\rm Pic}(X) = {\rm Div}(X)/{\rm Pr}(X) \text{ and } {\rm Pic}^{0}(X) = {\rm Div}^{0}(X)/{\rm Pr}(X).$$
We shall refer to ${\rm Pic}(X)$ as the Picard group of $X$, and to ${\rm Pic}^{0}(X)$ as the Picard group of degree zero of $X$.
\begin{theorem}
If $X$ is a finite connected graph, then one has a short exact sequence of abelian groups
\begin{equation} \label{ses2}
0 \longrightarrow \mathbb{Z}\sum_{v \in V_{X}}v \longrightarrow {\rm Div}(X) \stackrel{\mathcal{L}_{X}}{\longrightarrow} {\rm Pr}(X) \longrightarrow 0,
\end{equation}
so that ${\rm rank}_{\mathbb{Z}}({\rm Pr}(X)) = {\rm rank}_{\mathbb{Z}}({\rm Div}^{0}(X)) = |V_{X}|-1$ from which it follows that ${\rm Pic}^{0}(X)$ is a finite abelian group.  Moreover $|{\rm Pic}^{0}(X)| = \kappa(X)$, where $\kappa(X)$ is the number of spanning trees of $X$.
\end{theorem}
\begin{proof}
On one hand, we have 
\begin{equation*}
\begin{aligned}
\mathcal{L}_{X}\left(\sum_{v \in V_{X}}v \right) &= \sum_{v \in V_{X}} {\rm val}_{X}(v)v - \sum_{v \in V_{X}} \sum_{e \in \mathbf{E}_{X,v}}t(e)\\
&= \sum_{v \in V_{X}} {\rm val}_{X}(v)v - \sum_{e \in \mathbf{E}_{X}} t(e) \\
&= \sum_{v \in V_{X}} {\rm val}_{X}(v)v - \sum_{v \in V_{X}} \sum_{e \in \mathbf{E}_{X,v}^{t}}v\\
&=0.
\end{aligned}
\end{equation*}
On the other hand, if $D = \sum_{v \in V_{X}} m_{v}v \in {\rm ker}(\mathcal{L}_{X})$, then a simple calculation shows that
$$m_{v} = \frac{1}{{\rm val}_{X}(v)} \sum_{e \in \mathbf{E}_{X,v}^{t}}m_{o(e)} $$
for all $v \in V_{X}$.  Choosing $v$ so that $m_{v}$ is maximal implies that $m_{o(e)} = m_{v}$ for all $e \in \mathbf{E}_{X,v}^{t}$, and the connectedness assumption allows us to deduce that 
$$D \in \mathbb{Z}\sum_{v \in V_{X}}v. $$
This shows that (\ref{ses2}) is a short exact sequence.  The fact that $|{\rm Pic}^{0}(X)| = \kappa(X)$ follows from a classical theorem of Kirchhoff in graph theory.  See any book on graph theory such as \cite{CP-book}.
\end{proof}

%\subsection{Groups acting on graphs}
\subsection{Groups acting on graphs} \label{group}
The group of automorphisms of a graph $Y$ will be denoted as usual by ${\rm Aut}(Y)$.  Throughout this paper, we will often have groups acting on graphs.  The precise definition is as follows.
\begin{definition}
Let $G$ be a group and $Y$ a graph.  We say that $G$ acts on $Y$ if we are given a group morphism $G \rightarrow {\rm Aut}(Y)$.  The group $G$ is said to act without inversion if for all $e \in \mathbf{E}_{Y}$ and all $\sigma \in G$, one has $\sigma \cdot e \neq \bar{e}$.
\end{definition}
If $G$ acts without inversion on a graph $Y$, then one obtains a graph $G \backslash Y$ by letting the vertices to be $G \backslash V_{Y}$ and the directed edges to be $G \backslash \mathbf{E}_{Y}$.  The incidence map is given by $o(G \cdot e) = G \cdot o(e)$, $t(G \cdot e) = G \cdot t(e)$, and the inversion map by $\overline{G \cdot e} =  G \cdot \bar{e}$.  We leave it to the reader to check that $G \backslash Y$ is a graph, and that the natural map $Y \rightarrow G \backslash Y$ given for $w \in V_{Y}$ and $e \in \mathbf{E}_{Y}$ by
$$w \mapsto G \cdot w \text{ and } e \mapsto G \cdot e  $$
is a morphism of graphs.

%\subsection{The Picard group of degree zero as a \texorpdfstring{$\mathbb{Z}[G]$}{}-module}
\subsection{The Picard group of degree zero as a \texorpdfstring{$\mathbb{Z}[G]$}{}-module} \label{equivariant}

Let $Y$ be a graph and assume that a group $G$ acts on $Y$ without inversion.  Then, ${\rm Div}(Y)$ is acted upon by $G$ as well, and the short exact sequence
$$0 \rightarrow {\rm Div}^{0}(Y) \rightarrow {\rm Div}(Y) \stackrel{s}{\rightarrow} \mathbb{Z} \rightarrow 0$$
becomes one of $\mathbb{Z}[G]$-modules, where $G$ acts trivially on $\mathbb{Z}$.  Note that every $\sigma \in G$ induces a bijection
\begin{equation} \label{bij}
\mathbf{E}_{Y,w} \rightarrow \mathbf{E}_{Y,\sigma(w)} 
\end{equation}
for all $w \in V_{Y}$, so that if $Y$ is locally finite, one has
\begin{equation} \label{preserve_valency}
{\rm val}_{Y}(\sigma \cdot w) = {\rm val}_{Y}(w),
\end{equation}
for all $w \in V_{Y}$ and all $\sigma \in G$.

\begin{proposition}
Assuming that $Y$ is locally finite, the degree, the adjacency, and the Laplacian operators on ${\rm Div}(Y)$ are all morphisms of $\mathbb{Z}[G]$-modules.
\end{proposition}    
\begin{proof}
The degree operator $\mathcal{D}_{Y}$ is $G$-equivariant because of (\ref{preserve_valency}) above.  The adjacency operator is $G$-equivariant, since
$$\sum_{e \in \mathbf{E}_{Y,\sigma \cdot w}} t(e) = \sigma \cdot \sum_{e \in \mathbf{E}_{Y,w}}t(e) $$
by (\ref{bij}) above.  It then follows that the Laplacian operator $\mathcal{L}_{Y}$ is $G$-equivariant as well, since $\mathcal{L}_{Y} = \mathcal{D}_{Y} - \mathcal{A}_{Y}$.
\end{proof}
As a direct consequence, we obtain the following result.
\begin{corollary} \label{G_module}
Let $Y$ be a locally finite graph, and let $G$ be a group acting on $Y$ without inversion.  Then, both ${\rm Pic}(Y)$ and ${\rm Pic}^{0}(Y)$ are $\mathbb{Z}[G]$-modules.
\end{corollary}

%\subsection{Tensoring with $\mathbb{Z}_{p}$}
\subsection{Tensoring with \texorpdfstring{$\mathbb{Z}_{p}$}{} over \texorpdfstring{$\mathbb{Z}$}{}} \label{tensoring_z_p}
Later on, we will be working over $\mathbb{Z}_{p}$, rather than over $\mathbb{Z}$, where $p$ is a fixed rational prime.  Since $\mathbb{Z}_{p}$ is flat over $\mathbb{Z}$, everything we did so far works equally well over $\mathbb{Z}_{p}$.  If $M$ is a $\mathbb{Z}$-module, then we let
$$M_{p} = \mathbb{Z}_{p} \otimes_{\mathbb{Z}}M. $$
For instance, if $X$ is a graph, then we have ${\rm Div}_{p}(X)$, ${\rm Div}_{p}^{0}(X), {\rm Pr}_{p}(X), {\rm Pic}_{p}(X)$, and ${\rm Pic}_{p}^{0}(X)$ whenever these groups are defined.  After tensoring the short exact sequence (\ref{ses_basic}) with $\mathbb{Z}_{p}$ over $\mathbb{Z}$, one gets
\begin{equation}  \label{ses_basic_p}
0 \longrightarrow {\rm Div}_{p}^{0}(X) \longrightarrow {\rm Div}_{p}(X) \stackrel{s_{p}}{\longrightarrow} \mathbb{Z}_{p} \longrightarrow 0. 
\end{equation}
If $X$ is locally finite, then the two short exact sequences
$$0 \rightarrow {\rm Pr}(X) \rightarrow {\rm Div}(X) \rightarrow {\rm Pic}(X) \rightarrow 0, $$
and
$$0 \rightarrow {\rm Pr}(X) \rightarrow {\rm Div}^{0}(X) \rightarrow {\rm Pic}^{0}(X) \rightarrow 0 $$
becomes after tensoring with $\mathbb{Z}_{p}$ over $\mathbb{Z}$ the short exact sequences
$$0 \rightarrow {\rm Pr}_{p}(X) \rightarrow {\rm Div}_{p}(X) \rightarrow {\rm Pic}_{p}(X) \rightarrow 0, $$
and
$$0 \rightarrow {\rm Pr}_{p}(X) \rightarrow {\rm Div}_{p}^{0}(X) \rightarrow {\rm Pic}_{p}^{0}(X) \rightarrow 0 $$
of $\mathbb{Z}_{p}$-modules.  If furthermore $X$ is finite and connected, then the short exact sequence (\ref{ses2}) becomes
$$0 \longrightarrow \mathbb{Z}_{p}\sum_{v \in V_{X}}v \longrightarrow {\rm Div}_{p}(X) \stackrel{\mathcal{L}_{p}}{\longrightarrow} {\rm Pr}_{p}(X) \longrightarrow 0. $$
In this situation, note that we have
$${\rm Pic}_{p}^{0}(X) \simeq {\rm Div}_{p}^{0}(X)/{\rm Pr}_{p}(X) \simeq {\rm Pic}^{0}(X)[p^{\infty}], $$
where ${\rm Pic}^{0}(X)[p^{\infty}]$ denotes the Sylow $p$-subgroup of the finite abelian group ${\rm Pic}^{0}(X)$.  In particular, we have
$$|{\rm Pic}_{p}^{0}(X)| = \kappa_{p}(X), $$
where $\kappa_{p}(X) = p^{{\rm ord}_{p}(\kappa(X))}$.  In other words, tensoring with $\mathbb{Z}_{p}$ over $\mathbb{Z}$ allows us to focus on the $p$-part of the invariant $\kappa(X)$.  If moreover the graph $X$ is acted upon by a group $G$, then all these $\mathbb{Z}_{p}$-modules become modules over the group ring $\mathbb{Z}_{p}[G]$.

%\subsection{The Iwasawa algebra and Iwasawa modules}
\subsection{The Iwasawa algebra and Iwasawa modules} \label{Iwasawa_algebra}
In this section, we gather some well-known results on the Iwasawa algebra $\Lambda$ of $\mathbb{Z}_{p}$ and finitely generated modules over $\Lambda$ that we will use throughout.  Our main references for this section are \cite{Bourbaki}, \cite{1}, and \cite{Neukirch:2008}.

We prefer to work with a multiplicative notation, so we let $\Gamma$ be a multiplicative group that is topologically isomorphic to $\mathbb{Z}_{p}$.  We fix once and for all a topological generator $\gamma$ for $\Gamma$.  For each integer $n \ge 0$, we set $\Gamma_{n} = \Gamma/\Gamma^{p^{n}}$ and we let also $\gamma_{n} = \gamma^{p^{n}}$.  The Iwasawa algebra $\Lambda$ is the profinite completion of the group ring $R = \mathbb{Z}_{p}[\Gamma]$, in other words
$$\Lambda = \mathbb{Z}_{p}\llbracket \Gamma \rrbracket = \varprojlim_{n \ge 0}\mathbb{Z}_{p}[\Gamma_{n} ], $$
where the compatible maps $\mathbb{Z}_{p}[\Gamma_{n+1}] \twoheadrightarrow \mathbb{Z}_{p}[\Gamma_{n}]$ are the natural projection maps.  It is a unique factorization domain that is a noetherian ring of Krull dimension two.  Moreover it is a local ring with unique maximal ideal given by $\mathfrak{m} = (p,\gamma-1)$.  The ring $\Lambda$ is a topological ring when endowed with the $\mathfrak{m}$-adic topology which is compact, and hence also complete.

The projection maps $\phi_{n}:\mathbb{Z}_{p}[\Gamma] \twoheadrightarrow \mathbb{Z}_{p}[\Gamma_{n}]$ induce a unital ring morphism $\mathbb{Z}_{p}[\Gamma] \rightarrow \Lambda$ which is injective, since $\bigcap_{n \ge 0}{\rm ker}(\phi_{n}) = 0$.  Thus, we have a natural embedding of unital commutative rings $R \hookrightarrow \Lambda$.  Throughout, we let 
$$\omega_{n} = \gamma_{n} - 1 = \gamma^{p^{n}} - 1 \in \mathfrak{m}^{n} \subseteq R \subseteq \Lambda. $$
The natural projection maps $\Lambda \twoheadrightarrow \mathbb{Z}_{p}[\Gamma_{n}]$ and $R \twoheadrightarrow \mathbb{Z}_{p}[\Gamma_{n}]$ induce isomorphisms 
\begin{equation} \label{proj}
\Lambda/\omega_{n}\Lambda \stackrel{\simeq}{\longrightarrow} \mathbb{Z}_{p}[\Gamma_{n}]  \text{ and } R/R \cap \omega_{n}\Lambda \stackrel{\simeq}{\longrightarrow} \mathbb{Z}_{p}[\Gamma_{n}] 
\end{equation}
of unital commutative rings.  In addition to the elements $\omega_{n}$, we will need the elements
$$\omega_{n,k} = \frac{\omega_{n}}{\omega_{k}} = 1 + \gamma_{k} + \gamma_{k}^{2} + \ldots + \gamma_{k}^{p^{n-k}-1} \in R, $$
whenever $k,n \in \mathbb{Z}_{\ge 0}$ are such that $k \le n$.  

Any $\Lambda$-module $M$ is in fact a topological $\Lambda$-module by taking for a basis of neighborhoods around $m \in M$, the sets $m + \mathfrak{m}^{n}\cdot M$, where $n \ge 0$.  With this topology, any morphism of $\Lambda$-modules is automatically continuous.  A finitely generated module over the Iwasawa algebra $\Lambda$ is called an Iwasawa module, and any such module is compact.  We can view $\mathbb{Z}_{p}$ as a $\Lambda$-module or an $R$-module with trivial action.  As such, we have in particular
\begin{equation} \label{trivial_module}
\mathbb{Z}_{p} \simeq \Lambda/\omega_{0}\Lambda \text{ and } \mathbb{Z}_{p} \simeq R/R \cap \omega_{0}\Lambda  .
\end{equation}  
A morphism of Iwasawa modules $f: M \rightarrow N$ is called a pseudo-isomorphism if it has finite kernel and cokernel.  This defines an equivalence relation on the collection of torsion Iwasawa modules, and we will write $M \sim N$ if $M$ and $N$ are pseudo-isomorphic torsion Iwasawa modules.  To every torsion Iwasawa module $M$ is associated its Iwasawa invariants $\mu(M)$ and $\lambda(M)$ which are non-negative integers.  Two pseudo-isomorphic torsion Iwasawa modules have the same Iwasawa invariants.  
\begin{theorem} \label{washington}
Let $M$ be an Iwasawa module and assume that there exists $k \in \mathbb{Z}_{\ge 0}$ such that $M/\omega_{n,k}M$ is finite for all $n \ge k$.  Then $M$ is a torsion Iwasawa module.  Moreover, there exist $n_{0}\in \mathbb{Z}_{\ge 0}$ and $\nu \in \mathbb{Z}$ such that
$${\rm ord}_{p}(|M/\omega_{n,k}M|) = \mu(M) p^{n} + \lambda(M) n + \nu, $$
when $n \ge n_{0}$.
\end{theorem}
\begin{proof}
See \cite[Theorem 13.19]{1} and \cite[Lemma 13.21]{1}.
\end{proof}
If $M$ is a torsion Iwasawa module, then we shall denote its characteristic ideal by ${\rm char}_{\Lambda}(M)$.  The characteristic ideal is also an invariant that depends only on the pseudo-isomorphism class of the torsion Iwasawa module.  If 
$$0 \rightarrow M_{1} \rightarrow M_{2} \rightarrow M_{3} \rightarrow 0 $$
is a short exact sequence of torsion Iwasawa modules, then one has
\begin{equation} \label{ideal_mult}
{\rm char}_{\Lambda}(M_{2}) = {\rm char}_{\Lambda}(M_{1}) \cdot {\rm char}_{\Lambda}(M_{3}).
\end{equation}
The characteristic ideal can sometimes be calculated as follows.
\begin{theorem} \label{char_ideal_calc}
Let $F$ be a free $\Lambda$-module of finite rank and let $M$ be a finitely generated torsion Iwasawa module.  If we have a short exact sequence
$$0 \rightarrow F \stackrel{f}{\longrightarrow} F \longrightarrow M \longrightarrow 0 $$
of $\Lambda$-modules, then
$${\rm char}_{\Lambda}(M) = ({\rm det}(f)). $$
\end{theorem}
\begin{proof}
See \cite[Chapitre VII, \S 4, Corollaire to Proposition 14]{Bourbaki}.
\end{proof}

There is a non-canonical isomorphism 
\begin{equation} \label{non_can}
\Lambda \stackrel{\simeq}{\longrightarrow} \mathbb{Z}_{p}\llbracket T \rrbracket
\end{equation}
given by $\gamma \mapsto 1 + T$.  Therefore, any $\Lambda$-module can be viewed as a $\mathbb{Z}_{p}\llbracket T \rrbracket$-module, and Iwasawa modules are in particular finitely generated $\mathbb{Z}_{p}\llbracket T \rrbracket$-modules.  Every Iwasawa module $M$ is pseudo-isomorphic to an Iwasawa module of the form
\begin{equation} \label{elem_module}
\mathbb{Z}_{p}\llbracket T \rrbracket^{r} \oplus \bigoplus_{i=1}^{s}\mathbb{Z}_{p}\llbracket T \rrbracket/(p^{m_{i}}) \oplus \bigoplus_{j=1}^{t}\mathbb{Z}_{p}\llbracket T \rrbracket/(f_{j}(T)^{n_{j}}), 
\end{equation}
where $r,s,t,m_{i},n_{j} \in \mathbb{Z}_{\ge 0}$ and the $f_{j}(T) \in \mathbb{Z}_{p}[T]$ are distinguished and irreducible polynomials.  Recall that a polynomial $f(T) \in \mathbb{Z}_{p}[T]$ is called distinguished if it is monic and $p$ divides every coefficient except the leading one.  For instance, the elements
$$\omega_{n}(T) = (1+T)^{p^{n}}-1 \in \mathbb{Z}[T] \subseteq \mathbb{Z}_{p}\llbracket T \rrbracket $$
obtained from the $\omega_{n}$ via the isomorphism (\ref{non_can}) are distinguished.  An Iwasawa module is torsion precisely when $r=0$ in (\ref{elem_module}) above.  If $M$ is a torsion Iwasawa module pseudo-isomorphic to an Iwasawa module of the form
$$\bigoplus_{i=1}^{s}\mathbb{Z}_{p}\llbracket T \rrbracket/(p^{m_{i}}) \oplus \bigoplus_{j=1}^{t}\mathbb{Z}_{p}\llbracket T \rrbracket/(f_{j}(T)^{n_{j}}), $$
then its characteristic ideal, as an ideal of $\mathbb{Z}_{p}\llbracket T \rrbracket$, is given by
$${\rm char}_{\mathbb{Z}_{p}\llbracket T \rrbracket}(M) = (f(T)), $$
where 
\begin{equation} \label{char_id}
f(T) = \prod_{i=1}^{s}p^{m_{i}} \cdot \prod_{j=1}^{t}f_{j}(T)^{n_{j}} \in \mathbb{Z}_{p}[T]. 
\end{equation}
The Iwasawa invariants of $M$ are given by
$$\mu(M) = \mu(f(T)) \text{ and } \lambda(M) = \lambda(f(T)), $$
where we recall that given any 
$$g(T) = a_{0} + a_{1}T + a_{2}T^{2} + \ldots \in \mathbb{Z}_{p}\llbracket T \rrbracket,$$ 
one sets
$$\mu(g(T))= {\rm min}\{{\rm ord}_{p}(a_{i}): i \ge 0 \},$$
and
$$\lambda(g(T)) = {\rm min}\{i \ge 0 : {\rm ord}_{p}(a_{i}) = \mu(g(T)) \}.$$
For the polynomial $f(T)$ of (\ref{char_id}) above, one has
$$\mu(f(T)) = \sum_{i=1}^{s}m_{i}, $$
and
$$\lambda(f(T)) = {\rm deg}(f(T)), $$
since the $f_{j}(T)$ are distinguished.

%\section{Branched covers of graphs}
\section{Branched covers of graphs} \label{branched_cov}
%\subsection{Branched covers}
\subsection{Branched covers}
The following definition is taken from \cite[page 69]{Sunada:2013}.
\begin{definition} \label{branched}
Let $X$ and $Y$ be graphs and let $p:Y \rightarrow X$ be a morphism of graphs.  Then $p$ is called a branched (or ramified) cover if the following two conditions are satisfied:
\begin{enumerate}
\item Both functions $p:V_{Y} \rightarrow V_{X}$ and $p:\mathbf{E}_{Y} \rightarrow \mathbf{E}_{X}$ are surjective,
\item For all $w \in V_{Y}$, the cardinality $m_{w}$, also denoted by $m_{p}(w)$, of $(p|_{\mathbf{E}_{Y,w}})^{-1}(e)$ is independent of $e \in \mathbf{E}_{X,p(w)}$.  The cardinality $m_{w}$ is called the ramification index of the vertex $w$.  If $m_{w}$ is a positive integer, then the function
$$p|_{\mathbf{E}_{Y,w}}:\mathbf{E}_{Y,w} \rightarrow \mathbf{E}_{X,p(w)} $$
is $m_{w}$-to-$1$.
\end{enumerate}
\end{definition}
If $p:Y \rightarrow X$ is a branched cover for which $m_{w}=1$ for all $w \in V_{Y}$, then we will refer to $p$ as an unramified cover or simply a cover.  In this case, the functions $p|_{\mathbf{E}_{Y,w}}$ are bijections for all $w \in V_{Y}$.  In the situation where both $X$ and $Y$ are locally finite and $m_{w}$ is finite, we have
\begin{equation}\label{valency_change}
{\rm val}_{Y}(w) = m_{w} \cdot {\rm val}_{X}(p(w)).
\end{equation}
If $X,Y,Z$ are graphs and $p:Z \rightarrow Y$, $q:Y \rightarrow X$ are both branched covers of graphs, then so is the composition $q \circ p:Z \rightarrow X$.  If $w$ is a vertex of $Z$, then one has 
$$m_{q \circ p}(w) = m_{p}(w) \cdot m_{q}(p(w)),$$
when all these quantities are finite.  As pointed out in \cite[page 69]{Sunada:2013}, a branched cover in the sense of Definition \ref{branched} is an example of a harmonic morphism as defined for instance in \cite{Urakawa:2000} and \cite{Baker/Norine:2009} under suitable assumptions on the graphs involved such as simplicity or looplessness.
\begin{remark} \label{conn}
If $p:Y \rightarrow X$ is a branched cover and $Y$ is connected, then so is $X$.  Indeed, if $v_{1}, v_{2} \in V_{X}$, then let $w_{1}, w_{2} \in V_{Y}$ be such that $p(w_{i})=v_{i}$ for $i=1,2$.  Since $Y$ is connected, there exists a path $c$ in $Y$ going from $w_{1}$ to $w_{2}$.  Then, the path $p(c)$ is a path in $X$ going from $v_{1}$ to $v_{2}$ showing the claim.
\end{remark}
Following \cite[page 69]{Sunada:2013}, given a branched cover $p:Y \rightarrow X$ and $v \in V_{X}$, let
$$d(v) = \sum_{w \in p^{-1}(v)} m_{w}, $$
provided this sum is finite.  Since we have commutative diagrams
\begin{equation*}
\begin{tikzcd}
\mathbf{E}_{Y,w}^{o} \arrow["p|_{\mathbf{E}_{Y,w}^{o}}",r] \arrow{d}[swap]{{\rm inv}}  &  \mathbf{E}_{X,p(w)}^{o} \arrow{d}{{\rm inv}}\\
\mathbf{E}_{Y,w}^{t} \arrow["p|_{\mathbf{E}_{Y,w}^{t}}",r] &  \mathbf{E}_{X,p(w)}^{t},
\end{tikzcd}
\end{equation*}
it follows that the maps
$$p|_{\mathbf{E}_{Y,w}^{t}}:\mathbf{E}_{Y,w}^{t} \rightarrow \mathbf{E}_{X,p(w)}^{t}$$
are also $m_{w}$-to-$1$.  Therefore, if $e \in \mathbf{E}_{X}$, then
$$d(o(e)) = |p^{-1}(e)| = d(t(e)), $$
and if $X$ is connected, $d(v)$ is independent of the vertex $v \in V_{X}$.  This nonnegative integer is called the degree of the branched cover $p$ and is denoted by $[Y:X]$.  One has
\begin{equation} \label{degree}
[Y:X] = \sum_{w \in p^{-1}(v)} m_{w},
\end{equation}
for all $v \in V_{X}$, when $X$ is assumed to be connected.

Whenever a group $G$ acts on a graph $Y$, one gets two $G$-sets, as both $V_{Y}$ and $\mathbf{E}_{Y}$ are acted upon by $G$ as well.  Note that if $G$ acts on a graph $Y$ in a way that $G$ acts freely on $V_{Y}$, then it necessarily acts freely on $\mathbf{E}_{Y}$ as well.  Indeed, if $\sigma \cdot e = e$ for some $e \in \mathbf{E}_{Y}$, then $o(e) = o(\sigma \cdot e) = \sigma \cdot o(e)$ from which we deduce that $\sigma$ is the neutral element of $G$ provided $G$ acts freely on $V_{Y}$.

\begin{proposition}
Let $Y$ be a graph and let $G$ be a group acting without inversion on $Y$.  
\begin{enumerate}
\item If $G$ acts freely on $\mathbf{E}_{Y}$, then the natural morphism of graphs $p:Y \rightarrow G \backslash Y$ is a branched cover.  Moreover, for all $w \in V_{Y}$, one has $m_{w} = |S_{w}|$, where $S_{w} = {\rm Stab}_{G}(w)$.
\item If $G$ acts freely on $V_{Y}$, then the natural morphism of graphs $p:Y \rightarrow G \backslash Y$ is an unramified cover.
\end{enumerate}
\end{proposition}
\begin{proof}
It is clear that $p$ is surjective both on vertices and directed edges.  For simplicity let $X = G \backslash Y$ and let $G \cdot e \in \mathbf{E}_{X,p(w)}$, where $w \in V_{Y}$ is an arbitrary vertex of $Y$.  Then $o(G \cdot e)= G \cdot w$, and there exists $\sigma \in G$ such that $o(\sigma \cdot e) = w$.  A simple calculation shows that
$$\left(p|_{\mathbf{E}_{Y,w}} \right)^{-1}(G \cdot e) = \{\tau \cdot \sigma \cdot e : \tau \in S_{w} \}. $$
Since $G$ acts freely on $\mathbf{E}_{Y}$, all elements $\tau \cdot \sigma \cdot e$ are distinct as $\tau$ runs over $S_{w}$.  This ends the proofs of both claims.
\end{proof}

%\subsection{The Picard group of degree zero and branched covers}
\subsection{The Picard group of degree zero and branched covers} \label{picard}
Suppose now that we have a branched cover of graphs $f:Y \rightarrow X$, then we obtain a natural surjective group morphism
$$f_{*}:{\rm Div}(Y) \rightarrow {\rm Div}(X) $$
given by $w \mapsto f_{*}(w) = f(w)$.  If moreover $m_{w}$ is finite for all $w \in V_{Y}$, then we have another natural group morphism 
$$f_{r}:{\rm Div}(Y) \rightarrow {\rm Div}(X)$$ 
given by $w \mapsto f_{r}(w) = m_{w}\cdot f(w)$.  
\begin{proposition} \label{compatible_pic}
Assuming that $X$ and $Y$ are locally finite, and with the notation as above, we have a commutative diagram
\begin{equation*}
\begin{tikzcd}
{\rm Div}(Y) \arrow["\mathcal{L}_{Y}",r] \arrow{d}[swap]{f_{r}}  &  {\rm Div}(Y) \arrow{d}{f_{*}}\\
{\rm Div}(X) \arrow["\mathcal{L}_{X}",r] &  {\rm Div}(X).
\end{tikzcd}
\end{equation*}
\end{proposition}
\begin{proof}
The commutativity of the diagram follows from the commutativity of the two diagrams
\begin{equation*}
\begin{tikzcd}
{\rm Div}(Y) \arrow["\mathcal{D}_{Y}",r] \arrow{d}[swap]{f_{r}}  &  {\rm Div}(Y) \arrow{d}{f_{*}}\\
{\rm Div}(X) \arrow["\mathcal{D}_{X}",r] &  {\rm Div}(X),
\end{tikzcd}
\text{ and }
\begin{tikzcd}
{\rm Div}(Y) \arrow["\mathcal{A}_{Y}",r] \arrow{d}[swap]{f_{r}}  &  {\rm Div}(Y) \arrow{d}{f_{*}}\\
{\rm Div}(X) \arrow["\mathcal{A}_{X}",r] &  {\rm Div}(X).
\end{tikzcd}
\end{equation*}
The first diagram commutes because of (\ref{valency_change}), and the second diagram commutes since $f|_{\mathbf{E}_{Y,w}}:\mathbf{E}_{Y,w} \rightarrow \mathbf{E}_{X,f(w)}$ is $m_{w}$-to-$1$.
\end{proof}
Since $f_{*}:{\rm Div}(Y) \rightarrow {\rm Div}(X)$ is surjective, it follows from Proposition \ref{compatible_pic} that we have two surjective group morphisms
\begin{equation} \label{induced_map}
f_{*}:{\rm Pic}(Y) \rightarrow {\rm Pic}(X) \text{ and } f_{*}:{\rm Pic}^{0}(Y) \rightarrow {\rm Pic}^{0}(X).
\end{equation}
The surjectivity of the second group morphism implies the divisibility $\kappa(X) \, | \, \kappa(Y)$ whenever $f:Y \rightarrow X$ is a branched cover of finite connected graphs.

\begin{remark}
The fact that $\kappa(X) \, | \, \kappa(Y)$ is known to hold true more generally for harmonic morphisms (see \cite[Section 4]{Baker/Norine:2009}).  
\end{remark}

%\section{Constructions of branched covers via voltage assignments}
\section{Constructions of branched covers via voltage assignments} \label{voltage}

%\subsection{The basic construction}
\subsection{The basic construction} \label{basic}
Let $X$ be a graph, $G$ a group (for which we use the multiplicative notation), and $\alpha:\mathbf{E}_{X} \rightarrow G$ a function satisfying
\begin{equation} \label{voltage_eq}
\alpha(\bar{e}) = \alpha(e)^{-1}.
\end{equation}
Such a function $\alpha:\mathbf{E}_{X} \rightarrow G$ satisfying (\ref{voltage_eq}) above is often called a voltage assignment on $X$ with values in the group $G$.  Note that if $S$ is an orientation for $X$, then it suffices to specify $\alpha$ on $S$ and set $\alpha(\bar{s}) = \alpha(s)^{-1}$ in order to get a function $\alpha:\mathbf{E}_{X} \rightarrow G$ satisfying (\ref{voltage_eq}) above.  Let also
$$\mathcal{I} = \{(v,I_{v}) \, | \, v \in V_{X} \text{ and } I_{v} \leq G \}, $$
be a collection of subgroups of $G$ indexed by the vertices of $X$.  We define a graph $X(G,\mathcal{I},\alpha)$ as follows.  The vertex set is the disjoint union 
$$V = \bigsqcup_{v \in V_{X}} \{ v\} \times G/I_{v}, $$
and the collection of directed edges is given by
$$\mathbf{E} = \mathbf{E}_{X} \times G. $$
The directed edge $(e,\sigma)$ connects the vertex $(o(e),\sigma I_{o(e)})$ to the vertex $(t(e),\sigma \alpha(e) I_{t(e)})$.  Furthermore, one lets
$$\overline{(e,\sigma)} = (\bar{e},\sigma \alpha(e)). $$
We leave it to the reader to check that $X(G,\mathcal{I},\alpha)$ is a graph.    

The group $G$ acts naturally on the graph $X(G,\mathcal{I},\alpha)$ without inversion.  Indeed, for simplicity, let $Y = X(G,\mathcal{I},\alpha).  $ If $\tau \in G$, let $\phi_{\tau}:Y \rightarrow Y$ be defined via
$$\phi_{\tau}(v,\sigma I_{v}) = (v,\tau \sigma I_{v}) \text{ and } \phi_{\tau}(e,\sigma) = (e,\tau \sigma). $$
We leave it to the reader to check that $\phi_{\tau}$ is an automorphism of graphs.  It is then simple to check that the map $G \rightarrow {\rm Aut}(Y)$ defined via $\tau \mapsto \phi_{\tau}$ is a homomorphism of groups which gives an action of $G$ on $Y$ without inversion.  It clearly acts freely on $\mathbf{E}_{Y}$, but not necessarily freely on $V_{Y}$.  In fact, if $w = (v,\sigma I_{v}) \in V_{Y}$, then
$$S_{w} = {\rm Stab}_{G}(w) = \sigma I_{v} \sigma^{-1}. $$

We define the map 
$$p:X(G,\mathcal{I},\alpha) \rightarrow X$$ 
via $p(v,\sigma \cdot I_{v}) = v$ and $p(e,\sigma) = e$.  It is simple to verify that $p$ is a morphism of graphs.  In fact, it is a branched cover, since for $w = (v,\sigma I_{v}) \in V_{Y}$, given $e \in \mathbf{E}_{X,v}$, one has
$$\left(p|_{\mathbf{E}_{Y,w}} \right)^{-1}(e) = \{ (e, \rho \cdot \sigma ) \, : \, \rho \in S_{w} = \sigma I_{v} \sigma^{-1}\}. $$
Thus the cardinality of $\left(p|_{\mathbf{E}_{Y,w}} \right)^{-1}(e)$ is $|I_{v}|$ and is independent of $e \in \mathbf{E}_{X,v}$.  Summarizing the previous discussion, we have the following proposition.
\begin{proposition} \label{explicit_cov}
Let $X$ be a graph and $\alpha:\mathbf{E}_{X} \rightarrow G$ a function satisfying (\ref{voltage_eq}).  Let also $\mathcal{I} = \{(v,I_{v}): I_{v} \le G \}$ be a collection of subgroups of $G$ indexed by $V_{X}$ and consider the graph $Y = X(G,\mathcal{I},\alpha)$.  Then, the natural morphism of graphs
$$p:Y \rightarrow X$$
is a branched cover of graphs.  The group $G$ acts naturally on $Y$ without inversion and freely on $\mathbf{E}_{Y}$.  Moreover, for each vertex $w \in V_{Y}$, the ramification index $m_{p}(w)$ depends only on $v = p(w) \in V_{X}$ and is equal to $|I_{v}|$.  If $G$ is finite, then (\ref{degree}) becomes
$$[Y:X] = r_{v} \cdot m_{v}, $$
where $[Y:X] = |G|$, $m_{v} = |I_{v}|$ and $r_{v} = |p^{-1}(v)|$ for all $v \in V_{X}$.
\end{proposition} 

%\subsection{Functoriality}
\subsection{Functoriality} \label{functoriality}
More generally, assume that $f:G_{1} \rightarrow G_{2}$ is a surjective group morphism and that $\alpha:\mathbf{E}_{X} \rightarrow G_{1}$ is a voltage assignment with values in $G_{1}$.  Composing with $f$ gives a voltage assignment $f \circ \alpha:\mathbf{E}_{X} \rightarrow G_{2}$ with values in $G_{2}$.  We let
$$f(\mathcal{I}) = \{(v,f(I_{v})): v \in V_{X} \},$$
so that $f(\mathcal{I})$ is a collection of subgroups of $G_{2}$ indexed by the vertices of $X$.  We obtain two graphs
$$X(G_{1},\mathcal{I},\alpha) \text{ and } X(G_{2},f(\mathcal{I}),f \circ \alpha), $$
and there is a natural map $f_{*}:X(G_{1},\mathcal{I},\alpha) \rightarrow X(G_{2},f(\mathcal{I}), f \circ \alpha)$ given by
$$f_{*}(v,\sigma_{1} I_{v}) = (v,f(\sigma_{1})f(I_{v})) \text{ and } f_{*}(e,\sigma_{1}) = (e,f(\sigma_{1})).$$
We leave it to the reader to check that $f_{*}:X(G_{1},\mathcal{I},\alpha) \rightarrow X(G_{2},f(\mathcal{I}), f \circ \alpha)$ is a morphism of graphs that is surjective on vertices and directed edges.  If fact it is a branched cover of graphs.  Indeed, for simplicity let as above $Z = X(G_{1},\mathcal{I},\alpha)$, but also $Y = X(G_{2},f(\mathcal{I}),f \circ \alpha)$.  Moreover, let $w = (v,\sigma_{1}I_{v})$ be a vertex of $Z$.  If $(e,\sigma_{2}) \in \mathbf{E}_{Y,f_{*}(w)}$, then $\sigma_{2} = f(\sigma_{1}\tau)$ for some $\tau \in I_{v}$.  A simple calculation shows that
\begin{equation} \label{ram_index_volt}
\left(f_{*}|_{\mathbf{E}_{Z,w}} \right)^{-1}((e,\sigma_{2})) = \{(e,\rho\sigma_{1}\tau) : \rho \in \sigma_{1} I_{v} \sigma_{1}^{-1} \cap {\rm ker}(f) \},
\end{equation}
so that the second condition of \cref{branched} is satisfied with 
$$m_{w} = |\sigma_{1} I_{v} \sigma_{1}^{-1} \cap {\rm ker}(f)|.$$
Note that if $f:G_{1} \rightarrow 1$ is the trivial group morphism, then we get back the situation of \cref{explicit_cov}.

%\subsection{Branched $\mathbb{Z}_{p}$-towers of graphs} 
\subsection{\texorpdfstring{Branched $\mathbb{Z}_{p}$}{}-towers of graphs} \label{towers}
We can now explain the $\mathbb{Z}_{p}$-towers of graphs that we will be studying in this paper.  As in \cref{Iwasawa_algebra}, we let $\Gamma$ be a multiplicative topological group isomorphic to $\mathbb{Z}_{p}$.  Let $X = (V_{X},\mathbf{E}_{X})$ be a finite connected graph and let $\alpha:\mathbf{E}_{X} \rightarrow \Gamma$ be a function satisfying (\ref{voltage_eq}).  Moreover, for each $v \in V_{X}$, choose a closed subgroup $I_{v}$ of $\Gamma$, and as in \cref{basic}, let $\mathcal{I} = \{(v,I_{v}): v \in V_{X} \}$.  We obtain a graph $X_{\infty} = X(\Gamma,\mathcal{I},\alpha)$ and a branched cover $X_{\infty} \rightarrow X$.  The graph $X_{\infty}$ is infinite.  In order to get finite graphs, consider for each integer $n \ge 1$ the natural surjective group morphism $\pi_{n}:\Gamma \twoheadrightarrow \Gamma_{n}$, where recall from \cref{Iwasawa_algebra} that we set $\Gamma_{n} = \Gamma/\Gamma^{p^{n}} \simeq \mathbb{Z}/p^{n}\mathbb{Z}$.  Let $\alpha_{n}:\mathbf{E}_{X} \rightarrow \Gamma_{n}$ be the composition $\pi_{n} \circ \alpha$.  We let also
$$\mathcal{I}_{n} = \{(v,\pi_{n}(I_{v})) : v \in V_{X} \} $$
which is a collection of subgroups of $\Gamma_{n}$ indexed by the vertices of $X$.  We obtain a family of finite graphs
$$X_{n} = X(\Gamma_{n},\mathcal{I}_{n},\alpha_{n}), $$
and the natural surjective group morphisms $\Gamma_{n+1} \twoheadrightarrow \Gamma_{n}$ induces by \cref{functoriality} branched covers $X_{n+1} \rightarrow X_{n}$ for each $n \ge 0$.  We thus obtain a tower of graphs
\begin{equation} \label{branched_tower}
X = X_{0} \leftarrow X_{1} \leftarrow X_{2} \leftarrow \ldots \leftarrow X_{n} \leftarrow \ldots,
\end{equation}
where each map $X_{n+1} \rightarrow X_{n}$ is a branched cover satisfying $[X_{n+1}:X_{n}] = p$.  We call such a tower a branched (or ramified) $\mathbb{Z}_{p}$-towers of finite graphs provided all $X_{n}$ are connected.  We explain a sufficient condition that guarantees the connectedness of all finite graphs $X_{n}$ in \cref{connectedness} below.  Each graph $X_{n}$ in the tower (\ref{branched_tower}) is acted upon by $\Gamma_{n}$ by \cref{G_module} and \cref{explicit_cov}.
\begin{proposition}\label{compatible}
For the purpose of this proposition, let $f$ denote the natural group morphism ${\rm Pic}^{0}(X_{n+1}) \twoheadrightarrow {\rm Pic}^{0}(X_{n})$ from (\ref{induced_map}) and let $p$ denote the natural projection map $\Gamma_{n+1} \twoheadrightarrow \Gamma_{n}$.  Then, for all $x \in {\rm Pic}^{0}(X_{n+1})$ and for all $\gamma \in \Gamma_{n+1}$, one has
$$p(\gamma) \cdot f(x) = f(\gamma \cdot x).$$
The same property holds true for the group morphism ${\rm Pic}(X_{n+1}) \twoheadrightarrow {\rm Pic}(X_{n})$.
\end{proposition}
\begin{proof}
Let $w = (v,\sigma \pi_{n+1}(I_{v}))$ be a vertex of $X_{n+1}$.  On one hand, we have
\begin{equation*}
\begin{aligned}
f(\gamma \cdot (v,\sigma\pi_{n+1}(I_{v}))) &= f(v,\gamma\sigma\pi_{n+1}(I_{v})) \\
&= (v,p(\gamma) p(\sigma) \pi_{n}(I_{v})),
\end{aligned}    
\end{equation*}
and on the other hand
\begin{equation*}
\begin{aligned}
p(\gamma) \cdot f(v,\sigma\pi_{n+1}(I_{v})) &= p(\gamma) (v,p(\sigma) \pi_{n}(I_{v})) \\
&= (v,p(\gamma) p(\sigma) \pi_{n}(I_{v}))
\end{aligned}    
\end{equation*}
showing the desired claim.
\end{proof}    
The group morphisms ${\rm Pic}^{0}(X_{n+1}) \rightarrow {\rm Pic}^{0}(X_{n})$ induce maps ${\rm Pic}^{0}_{p}(X_{n+1}) \rightarrow {\rm Pic}^{0}_{p}(X_{n})$ that form a compatible system of $\mathbb{Z}_{p}$-module morphisms.  By \cref{compatible}, the projective limit
\begin{equation} \label{Iwa_module_zero}
{\rm Pic}^{0}_{\Lambda} = \varprojlim_{n \ge 0} {\rm Pic}^{0}_{p}(X_{n}), 
\end{equation}
is a $\Lambda$-module.  Similarly, the group morphisms ${\rm Pic}(X_{n+1}) \rightarrow {\rm Pic}(X_{n})$ induce maps after tensoring with $\mathbb{Z}_{p}$ over $\mathbb{Z}$ that also form a compatible system of $\mathbb{Z}_{p}$-module morphisms ${\rm Pic}_{p}(X_{n+1}) \rightarrow {\rm Pic}_{p}(X_{n})$.  Thus, we obtain another $\Lambda$-module
\begin{equation} \label{Iwa_module}
{\rm Pic}_{\Lambda} = \varprojlim_{n \ge 0} {\rm Pic}_{p}(X_{n}).
\end{equation}

%\subsection{Immersions}
\subsection{Immersions} \label{immersions}

In this paper, we will encounter another type of morphisms for which we now give the precise definition.  
\begin{definition} \label{immersion}
Let $X$ and $Y$ be graphs.  A morphism of graphs $\iota:X \rightarrow Y$ is called an immersion if the induced function
$$\iota|_{\mathbf{E}_{X,v}} :\mathbf{E}_{X,v} \rightarrow \mathbf{E}_{Y,\iota(v)} $$
is injective for all $v \in V_{X}$.
\end{definition}

Starting with a graph $X$, a voltage assignment $\alpha:\mathbf{E}_{X} \rightarrow G$, and a collection of subgroups $\mathcal{I}$ indexed by the vertices of $X$ as in \cref{basic}, one can forget about $\mathcal{I}$ (or take $I_{v} = 1$ for all vertices $v \in V_{X}$), and keep the same voltage assignment $\alpha$.  In this case, the morphism of graphs $p:X(G,\alpha) \rightarrow X$ is an unramified cover of graphs.  Furthermore, we have a commutative diagram of the form
\begin{equation*}
\begin{tikzcd}
X(G,\alpha) \arrow["\iota",r] \arrow{d}[swap]{}  &   X(G,\mathcal{I},\alpha)\arrow{dl}{p}\\
X
\end{tikzcd}
\end{equation*}
where $\iota$ is the morphism of graphs given by 
$$(v,\sigma) \mapsto \iota(v,\sigma) = (v,\sigma I_{v}) \text{ and } (e,\sigma) \mapsto \iota(e,\sigma) = (e,\sigma). $$
Since $\iota$ is a bijection on directed edges, the map $\iota$ is an immersion of graphs.  Note that the vertical map on the left is an unramified cover of graphs, whereas the map $p$ is a branched cover as we explained above.  The morphism of graphs $\iota$ is clearly $G$-equivariant.

%\subsection{Connectedness}
\subsection{Connectedness} \label{connectedness}
We keep the same notation as in \cref{immersions}.

\begin{lemma} \label{conn_unr}
If $X(G,\alpha)$ is connected, then so is $X(G,\mathcal{I},\alpha)$.
\end{lemma} 
\begin{proof}
Let $(v_{1},\sigma_{1}I_{v_{1}})$ and $(v_{2},\sigma_{2}I_{v_{2}})$ be two vertices of $X(G,\mathcal{I},\alpha)$.  Since $X(G,\alpha)$ is assumed to be connected, there exists a path $c$ in $X(G,\alpha)$ going from $(v_{1},\sigma_{1})$ to $(v_{2},\sigma_{2})$.  Then, the path $\iota(c)$ of $X(G,\mathcal{I},\alpha)$ goes from $(v_{1},\sigma_{1}I_{v_{1}})$ to $(v_{2},\sigma_{2}I_{v_{2}})$ showing the claim.
\end{proof}   

Let now $\alpha:\mathbf{E}_{X} \rightarrow \Gamma$ be a voltage assignment and let $\mathcal{I}$ be a collection of closed subgroups of $\Gamma$ indexed by $V_{X}$ as in \cref{towers}.  If we let $X_{n}^{unr} = X(\Gamma_{n},\alpha_{n})$, then we have an unramified $\mathbb{Z}_{p}$-tower of finite graphs
$$X = X_{0} \leftarrow X_{1}^{unr} \leftarrow X_{2}^{unr} \leftarrow \ldots \leftarrow X_{n}^{unr} \leftarrow \ldots$$
that fits into a commutative diagram
\begin{equation*}
\begin{tikzcd}
& \arrow[ld] X_{1}^{unr}  \arrow[d, "\iota_{1}"]  & \arrow[l] X_{2}^{unr}   \arrow[d, "\iota_{2}"] &\arrow[l] \ldots  &\arrow[l] \arrow[d, "\iota_{n}"]  X_{n}^{unr} & \arrow[l]  \ldots\\  
X & \arrow[l]X_{1}  & \arrow[l]X_{2}   & \arrow[l] \ldots  & \arrow[l] X_{n}  \arrow[l] & \arrow[l] \ldots, \\
\end{tikzcd}
\end{equation*}
where the vertical maps are all immersions of graphs.  Letting $\pi_{1}(X,v_{0})$ be the fundamental group of $X$ based at a vertex $v_{0}$, the function $\alpha:\mathbf{E}_{X} \rightarrow \Gamma$ induces a group morphism
$$\rho_{\alpha}:\pi_{1}(X,v_{0}) \rightarrow \Gamma. $$
\begin{theorem} \label{connect}
With the notation as above, all the graphs $X_{n}^{unr}$ are connected if and only if $\rho_{\alpha}(\pi_{1}(X,v_{0}))$ generates $\Gamma$ topologically.
\end{theorem} 
\begin{proof}
This follows from \cite[Theorem 2.11]{Ray/Vallieres:2023}.
\end{proof}
Combined with Lemma \ref{conn_unr} and the discussion in Section $2.3.1$ of \cite{Ray/Vallieres:2023}, one has a sufficient condition that can be checked explicitly to construct branched $\mathbb{Z}_{p}$-towers of graphs, and guarantee the connectedness of all the graphs $X_{n}$.

In \cref{main_conj1} below, we will also consider the infinite graph $X_{\infty}^{unr} = X(\Gamma,\alpha)$.  Note that $X_{\infty}^{unr}$ is locally finite, whereas $X_{\infty}$ is not necessarily.

%\subsection{A useful lemma}
\subsection{A useful lemma} \label{useful_lemma}
Let us keep again the same notation as in \cref{immersions}.  The following lemma and corollary will be used in \cref{iwasawa_theory} below.
\begin{lemma} \label{tech}
To simplify the notation, let $Y^{unr} = X(G,\alpha)$, $Y = X(G,\mathcal{I},\alpha)$, and assume that $X,Y$ and $Y^{unr}$ are all locally finite graphs.  Since $\iota: Y^{unr} \rightarrow Y$ is surjective on vertices, we obtain a natural surjective group morphism
$$\iota_{*}:{\rm Div}(Y^{unr}) \rightarrow {\rm Div}(Y).$$
Assuming that $I_{v}$ is finite for all $v \in V_{X}$, then for all $v \in V_{X}$ and for all $\sigma \in G$, one has
$$\mathcal{L}_{Y} \circ \iota_{*}((v,\sigma)) = \sum_{i \in I_{v}} \iota_{*} \circ \mathcal{L}_{Y^{unr}}((v,\sigma \cdot i)). $$
\end{lemma}
\begin{proof}
On one hand, we have
\begin{equation*}
\begin{aligned}
\mathcal{L}_{Y} \circ \iota_{*}((v,\sigma)) &= \mathcal{L}_{Y}((v,\sigma I_{v})) \\ 
&=\mathcal{D}_{Y}((v,\sigma I_{v})) - \mathcal{A}_{Y}((v,\sigma I_{v})) \\
&=|I_{v}| {\rm val}_{X}(v) (v,\sigma I_{v}) - \sum_{\varepsilon \in \mathbf{E}_{Y,(v,\sigma I_{v})}} t(\varepsilon) \\
&=|I_{v}| {\rm val}_{X}(v) (v,\sigma I_{v}) - \sum_{e \in \mathbf{E}_{X,v}} \sum_{\tau \in \sigma I_{v}}(t(e),\tau \alpha(e) I_{t(e)}) \\
&=|I_{v}| {\rm val}_{X}(v) (v,\sigma I_{v}) - \sum_{e \in \mathbf{E}_{X,v}} \sum_{i \in I_{v}}(t(e),\sigma i \alpha(e) I_{t(e)}).
\end{aligned}
\end{equation*}
On the other hand, we have
\begin{equation*}
\begin{aligned}
\mathcal{L}_{Y^{unr}}((v,\sigma i)) &= {\rm val}_{X}(v) (v,\sigma i) - \sum_{\varepsilon \in \mathbf{E}_{Y^{unr},(v,\sigma i)}}t(\varepsilon) \\ 
&= {\rm val}_{X}(v) (v,\sigma i) - \sum_{e \in \mathbf{E}_{X,v}}(t(e),\sigma i \alpha(e)).
\end{aligned}
\end{equation*}
Therefore,
$$\iota_{*} \circ \mathcal{L}_{Y^{unr}}((v,\sigma i)) = {\rm val}_{X}(v)(v,\sigma I_{v}) - \sum_{e \in \mathbf{E}_{X,v}}(t(e),\sigma i \alpha(e) I_{t(e)})$$
from which it follows that
$$\sum_{i \in I_{v}} \iota_{*} \circ \mathcal{L}_{Y^{unr}}((v,\sigma i)) = |I_{v}| {\rm val}_{X}(v) (v,\sigma I_{v}) - \sum_{e \in \mathbf{E}_{X,v}} \sum_{i \in I_{v}}(t(e),\sigma i \alpha(e) I_{t(e)}),$$
and this ends the proof.
\end{proof}
\begin{corollary} \label{generators1}
With the same notation as in \cref{tech}, for each $v \in V_{X}$, let
$$P_{v} = \iota_{*} \circ \mathcal{L}_{Y^{unr}}((v,1_{G})) \in {\rm Div}(Y). $$
Then ${\rm Pr}(Y)$ is generated over $\mathbb{Z}[G]$ by 
$$\sum_{i \in I_{v}} i P_{v} $$
as $v$ runs over all vertices in $X$.
\end{corollary}
\begin{proof}
By definition, ${\rm Pr}(Y)$ is generated over $\mathbb{Z}[G]$ by the divisors $\mathcal{L}_{Y} \circ \iota_{*}((v,1_{G}))$ as $v$ runs through $V_{X}$.  The result then follows directly from \cref{tech}.
\end{proof}

%\section{Iwasawa theory for branched $\mathbb{Z}_{p}$towers of graphs}
\section{Iwasawa theory for branched \texorpdfstring{$\mathbb{Z}_{p}$}{}-towers of graphs} \label{iwasawa_theory}
We keep the same notation as before.  So $X=(V_{X},\mathbf{E}_{X})$ is a finite connected graph, $\Gamma$ is a multiplicative group that is topologically isomorphic to $\mathbb{Z}_{p}$ and $\alpha:\mathbf{E}_{X}\xrightarrow[]{} \Gamma$ is a voltage assignment.  For each $v\in V_{X}$, we pick a closed subgroup $I_v$ of $\Gamma$, 
and this gives rise to a branched $\mathbb{Z}_{p}$-tower of graphs
$$X=X_0 \xleftarrow[]{} X_1 \xleftarrow[]{} X_2 \xleftarrow[]{} ... \xleftarrow[]{} X_n \xleftarrow[]{}...$$
as explained \cref{towers}.  Recall that we always assume all of the graphs $X_{n}$ to be connected.  See \cref{connectedness} for a sufficient criterion that guarantees this condition is satisfied.

Our goal in this section is to study the $\Lambda$-module ${\rm Pic}^{0}_{\Lambda}$ that was defined above in (\ref{Iwa_module_zero}).  In order to do so, we will use the infinite graph $X_{\infty} = X(\Gamma,\mathcal{I},\alpha)$.  The short exact sequence (\ref{ses_basic_p}) for the graph $X_{\infty}$ gives a short exact sequence
$$0 \longrightarrow {\rm Div}^{0}_{p}(X_{\infty}) \longrightarrow {\rm Div}_{p}(X_{\infty}) \stackrel{s_{p}}{\longrightarrow} \mathbb{Z}_{p}\longrightarrow  0 $$
of $\mathbb{Z}_{p}[\Gamma]$-modules, since $\Gamma$ is acting on all of these modules, the action of $\Gamma$ on $\mathbb{Z}_{p}$ being the trivial one.  Recall from \cref{Iwasawa_algebra} that we set $R = \mathbb{Z}_{p}[\Gamma]$.  Letting
$${\rm Div}_{\Lambda}(X_{\infty}) = \Lambda \otimes_{R} {\rm Div}_{p}(X_{\infty}), $$
the surjective $R$-module morphism $s_{p}:{\rm Div}_{p}(X_{\infty}) \rightarrow \mathbb{Z}_{p}$ induces a surjective morphism
\begin{equation} \label{surj_}
{\rm id} \otimes s_{p}: {\rm Div}_{\Lambda}(X_{\infty}) \rightarrow \Lambda \otimes_{R}\mathbb{Z}_{p} 
\end{equation}
of $\Lambda$-modules.  We let 
\begin{equation} \label{degree_zero}
{\rm Div}_{\Lambda}^{0}(X_{\infty}) = {\rm ker}({\rm id} \otimes s_{p}).
\end{equation}
It follows from (\ref{trivial_module}) that 
$$\Lambda \otimes_{R} \mathbb{Z}_{p} \simeq \Lambda \otimes_{R} (R/R \cap \omega_{0}\Lambda) \simeq \Lambda/\omega_{0}\Lambda \simeq \mathbb{Z}_{p}$$
as $\Lambda$-modules, where the element $\omega_{0} = \gamma_{0} - 1 \in R \subseteq \Lambda$ was defined in \cref{Iwasawa_algebra}.  After making this identification, (\ref{surj_}) and (\ref{degree_zero}) gives us a short exact sequence
\begin{equation} \label{useful_ses}
0 \longrightarrow {\rm Div}_{\Lambda}^{0}(X_{\infty}) \longrightarrow {\rm Div}_{\Lambda}(X_{\infty}) \stackrel{s_{\Lambda}}{\longrightarrow} \mathbb{Z}_{p} \rightarrow 0 
\end{equation}
of $\Lambda$-modules.  At the finite level, for each $n \ge 0$ we also have a short exact sequence
\begin{equation} \label{div_finite}
0 \rightarrow {\rm Div}_{p}^{0}(X_{n}) \rightarrow {\rm Div}_{p}(X_{n}) \rightarrow \mathbb{Z}_{p} \rightarrow 0
\end{equation}
of $\mathbb{Z}_{p}[\Gamma_{n}]$-modules, since $\Gamma_{n}$ acts on all of these modules, the action of $\Gamma_{n}$ on $\mathbb{Z}_{p}$ being the trivial one.  Via the natural projection map $\Lambda \twoheadrightarrow \Lambda/\omega_{n}\Lambda \simeq \mathbb{Z}_{p}[\Gamma_{n}]$, we can actually view the short exact sequence (\ref{div_finite}) as one of $\Lambda$-modules as well.  The branched cover $\pi_{n}:X_{\infty} \rightarrow X_{n}$ induces a natural surjective morphism
$$\pi_{n}:{\rm Div}_{p}(X_{\infty}) \rightarrow {\rm Div}_{p}(X_{n}) $$
of $R$-modules.  The group ${\rm Div}_{p}(X_{\infty})$ is only a module over $R$, but since ${\rm Div}_{p}(X_{n})$ is a module over $\Lambda$, we obtain a natural surjective morphism
\begin{equation} \label{essential}
\pi_{n}:{\rm Div}_{\Lambda}(X_{\infty}) \rightarrow {\rm Div}_{p}(X_{n}) 
\end{equation}
of $\Lambda$-modules which we denote by the same symbol.  This map induces the commutative diagram of $\Lambda$-modules
\begin{equation*}
\begin{tikzcd}
0 \arrow[r] & {\rm Div}_{\Lambda}^{0}(X_{\infty}) \arrow[d] \arrow[r] & {\rm Div}_{\Lambda}(X_{\infty}) \arrow[d,"\pi_{n}"]\arrow[r,"s_{\Lambda}"] & \arrow[d]\arrow[r] \mathbb{Z}_{p} & 0 \\
0 \arrow[r] & {\rm Div}_{p}^{0}(X_{n}) \arrow[r] & {\rm Div}_{p}(X_{n}) \arrow[r,"s_{p}"] & \arrow[r] \mathbb{Z}_{p} & 0,
\end{tikzcd}
\end{equation*}
where the right vertical map is the identity, and both horizontal sequences are exact.

The structure of ${\rm Div}_{\Lambda}(X_{\infty})$ can be understood as follows.  For simplicity, we write $V$ instead of $V_{X}$.  We let $V^{ram}$ be the set of ramified vertices, that is the collection of vertices $v \in V$ for which $I_{v}$ is non-trivial.  We also let $V^{unr} = V \smallsetminus V^{ram}$.  From now on, we introduce a labeling of $V$, say $V = \{v_{1},\ldots, v_{s} \}$, where we agree that 
$$V^{unr} = \{v_{1},\ldots,v_{r} \} \text{ and } V^{ram} = \{v_{r+1},\ldots,v_{s} \}.$$  
For each $i=1,\ldots,s$, we let $w_{i,\infty}$ be the vertex $(v_{i},I_{i})$ of $V_{X_{\infty}}$, where we write $I_{i}$ instead of $I_{v_{i}}$.  Then, as an $R$-module we have
\begin{equation} \label{R_struct}
{\rm Div}_{p}(X_{\infty}) = \bigoplus_{i=1}^{r} \mathbb{Z}_{p}[\Gamma] \cdot w_{i,\infty} \oplus \bigoplus_{i=r+1}^{s} \mathbb{Z}_{p}[\Gamma/I_{i}] \cdot w_{i,\infty}.
\end{equation}
From now on, we let $k_{i}$ be the non-negative integer such that $I_{i} = \Gamma^{p^{k_{i}}}$. Note that $\mathbb{Z}_{p}[\Gamma/I_{i}] \simeq \mathbb{Z}_{p}[\Gamma_{k_{i}}]$, and we have isomorphisms of $\Lambda$-modules
$$\Lambda \otimes_{R} \mathbb{Z}_{p}[\Gamma/I_{i}] \simeq \Lambda \otimes_{R} \mathbb{Z}_{p}[\Gamma_{k_{i}}] \simeq \Lambda \otimes_{R} (R/R \cap \omega_{k_{i}}\Lambda) \simeq \Lambda/\omega_{k_{i}}\Lambda \simeq \mathbb{Z}_{p}[\Gamma_{k_{i}}]. $$
It follows that tensoring (\ref{R_struct}) with $\Lambda$ over $R$ gives
\begin{equation} \label{struc_mod_inf}
\begin{aligned}
{\rm Div}_{\Lambda}(X_{\infty}) &= \bigoplus_{i=1}^{r} \Lambda \cdot w_{i,\infty} \oplus \bigoplus_{i=r+1}^{s} \mathbb{Z}_{p}[\Gamma_{k_{i}}] \cdot w_{i,\infty} \\
& \simeq \Lambda^{r} \oplus \bigoplus_{i=r+1}^{s} \mathbb{Z}_{p}[\Gamma_{k_{i}}].
\end{aligned}
\end{equation}
Similarly, at the finite level we set
$$I_{i,n} = I_{i} \cdot \Gamma^{p^{n}}/\Gamma^{p^{n}} $$
which is the image of $I_{i}$ via the natural projection map $\pi_{n}:\Gamma \twoheadrightarrow \Gamma_{n}$, and we let $w_{i,n}$ be the vertex $(v_{i},I_{i,n})$ of $X_{n}$.  We then have
$${\rm Div}_{p}(X_{n}) = \bigoplus_{i=1}^{r}\mathbb{Z}[\Gamma_{n}] \cdot w_{i,n} \oplus \bigoplus_{i=r+1}^{s} \mathbb{Z}_{p}[\Gamma_{n}/I_{i,n}] \cdot w_{i,n}$$
as $\mathbb{Z}_{p}[\Gamma_{n}]$-modules.  Note that if $n \ge k_{i}$, then $I_{i,n} = \Gamma^{p^{k_{i}}}/\Gamma^{p^{n}}$.  
\begin{assumption} \label{n_large}
From now on, we assume that $n$ is large enough so that 
$$I_{i,n} = \Gamma^{p^{k_{i}}}/\Gamma^{p^{n}}$$ 
for all $i=r+1,\ldots,s$.
\end{assumption}
Under \cref{n_large}, we have
\begin{equation} \label{struc_mod_finite}
\begin{aligned}
{\rm Div}_{p}(X_{n}) &= \bigoplus_{i=1}^{r}\mathbb{Z}[\Gamma_{n}] \cdot w_{i,n} \oplus \bigoplus_{i=r+1}^{s} \mathbb{Z}_{p}[\Gamma_{k_{i}}] \cdot w_{i,n}\\
&\simeq \mathbb{Z}[\Gamma_{n}]^{r} \oplus \bigoplus_{i=r+1}^{s} \mathbb{Z}_{p}[\Gamma_{k_{i}}],
\end{aligned}
\end{equation}
since $\Gamma_{n}/I_{i,n} \simeq \Gamma/I_{i} \simeq \Gamma_{k_{i}}$.  

Note that from (\ref{struc_mod_inf}) and (\ref{struc_mod_finite}), the $\Lambda$-module morphism (\ref{essential}) is given by the natural projection map $\Lambda \twoheadrightarrow \mathbb{Z}_{p}[\Gamma_{n}]$ on the unramified component and by the identity map on the ramified component.
\begin{lemma} \label{kernel1}
The kernel of the $\Lambda$-module morphism $\pi_{n}$ from (\ref{essential}) is given by 
$${\rm ker}(\pi_{n}) = \omega_{n} {\rm Div}_{\Lambda}(X_{\infty}).$$
\end{lemma}
\begin{proof}
From (\ref{struc_mod_inf}) and (\ref{struc_mod_finite}), we get a short exact sequence 
$$0 \longrightarrow \bigoplus_{i=1}^{r} \omega_{n} \Lambda \cdot w_{i,\infty} \longrightarrow {\rm Div}_{\Lambda}(X_{\infty}) \stackrel{\pi_{n}}{\longrightarrow} {\rm Div}_{p}(X_{n}) \longrightarrow 0 $$
of $\Lambda$-modules.  Because of \cref{n_large}, we have $\omega_{k_{i}}\, | \, \omega_{n}$.  Therefore, 
$$\omega_{n}  \mathbb{Z}_{p}[\Gamma_{k_{i}}] \simeq \omega_{n} (\Lambda/\omega_{k_{i}}\Lambda) = 0. $$
It follows that 
$$\omega_{n} {\rm Div}_{\Lambda}(X_{\infty}) = \bigoplus_{i=1}^{r} \omega_{n} \Lambda \cdot w_{i,\infty}, $$
and the result follows.
\end{proof}
Note that it follows from (\ref{R_struct}) and (\ref{struc_mod_inf}) that the natural morphism of $R$-modules
$${\rm Div}_{p}(X_{\infty}) \rightarrow {\rm Div}_{\Lambda}(X_{\infty}) $$
is injective, so that we can view divisors in ${\rm Div}_{p}(X_{\infty})$ as lying in ${\rm Div}_{\Lambda}(X_{\infty})$.
\begin{definition}
For each $i=1,\ldots,s$, we let
\begin{equation} \label{generators}
P_{i,\infty} = {\rm val}_{X}(v_{i}) w_{i,\infty} - \sum_{e \in \mathbf{E}_{X,v_{i}}}(t(e),\alpha(e)I_{t(e)}) \in {\rm Div}_{\Lambda}(X_{\infty}). 
\end{equation}
Moreover, we define ${\rm Pr}_{\Lambda}^{unr}$ to be the $\Lambda$-submodule of ${\rm Div}_{\Lambda}(X_{\infty})$ generated by  
$$\{P_{i,\infty}: i=1,\ldots,r \},$$ 
and we let ${\rm Pr}_{\Lambda,n}^{ram}$ be the $\Lambda$-submodule of ${\rm Div}_{\Lambda}(X_{\infty})$ generated by
$$\left \{\omega_{n,k_{i}} P_{i,\infty}: i = r+1,\ldots,s\right \}, $$
where the elements $\omega_{n,k_{i}}$ were defined in \cref{Iwasawa_algebra}. 
\end{definition}

\begin{theorem} \label{finite_level}
Under \cref{n_large}, the morphism $\pi_{n}:{\rm Div}_{\Lambda}(X_{\infty}) \rightarrow {\rm Div}_{p}(X_{n})$ induces isomorphisms
$${\rm Div}_{\Lambda}(X_{\infty})/N_{n} \stackrel{\simeq}{\longrightarrow} {\rm Pic}_{p}(X_{n}) \text{ and } {\rm Div}_{\Lambda}^{0}(X_{\infty})/N_{n} \stackrel{\simeq}{\longrightarrow} {\rm Pic}_{p}^{0}(X_{n}), $$
of $\Lambda$-modules, where
$$N_{n} = \omega_{n} {\rm Div}_{\Lambda}(X_{\infty}) + {\rm Pr}_{\Lambda,n}^{ram} + {\rm Pr}_{\Lambda}^{unr}. $$
\end{theorem}
\begin{proof}
For each $n \ge 0$, let $P_{i,n} \in {\rm Div}_{p}(X_{n})$ be the divisor $P_{v_{i}}$ for the graph $X_{n}$ defined in \cref{generators1}.  Observe now that for $i=1,\ldots,r$, one has
$$\pi_{n}(P_{i,\infty}) = P_{i,n}, $$
whereas for $i=r+1,\ldots,s$, one has
\begin{equation*}
\begin{aligned}
\pi_{n}\left(\omega_{n,k_{i}}P_{i,\infty}\right) &=  \sum_{j=0}^{p^{n - k_{i}}-1}\gamma^{jp^{k_{i}}} \cdot P_{i,n} \\
&= \sum_{\sigma \in I_{i,n}} \sigma P_{i,n}.
\end{aligned}
\end{equation*}
\cref{generators1} and \cref{kernel1} imply the equality
\begin{equation} \label{equality}
\pi_{n}(N_{n}) = {\rm Pr}_{p}(X_{n}) 
\end{equation}
which implies also the inclusion $N_{n} \subseteq \pi_{n}^{-1}({\rm Pr}_{p}(X_{n}))$.  Conversely, if $D \in \pi_{n}^{-1}({\rm Pr}_{p}(X_{n}))$, then $\pi_{n}(D) \in {\rm Pr}_{p}(X_{n})$ so that by (\ref{equality}), one has $\pi_{n}(D) = \pi_{n}(D_{0})$ for some $D_{0} \in N_{n}$. \cref{kernel1} implies that $D - D_{0} \in \omega_{n}{\rm Div}_{\Lambda}(X_{\infty}) \subseteq N_{n}$ and this shows the equality
$$N_{n} = \pi_{n}^{-1}({\rm Pr}_{p}(X_{n})).$$
The first isomorphism then follows.  

For the second one, it suffices to notice that $\omega_{n} {\rm Div}_{\Lambda}(X_{\infty}) \subseteq {\rm Div}_{\Lambda}^{0}(X_{\infty})$, and that from the definition of $P_{i,\infty}$ above in (\ref{generators}), we also clearly have $P_{i,\infty} \in {\rm Div}_{\Lambda}^{0}(X_{\infty})$ for all $i=1,\ldots,s$.  Therefore, $N_{n} \subseteq {\rm Div}_{\Lambda}^{0}(X_{\infty})$ for all $n \ge 0$, and this concludes the proof.
\end{proof}
As a consequence, we obtain the following concrete descriptions of ${\rm Pic}_{\Lambda}$ and ${\rm Pic}_{\Lambda}^{0}$ that were defined above in (\ref{Iwa_module}) and (\ref{Iwa_module_zero}). 
\begin{corollary} \label{exp_desc}
With the same notation as above, one has
$${\rm Pic}_{\Lambda} \simeq {\rm Div}_{\Lambda}(X_{\infty})/{\rm Pr}_{\Lambda}^{unr} \text{ and } {\rm Pic}_{\Lambda}^{0} \simeq {\rm Div}_{\Lambda}^{0}(X_{\infty})/{\rm Pr}_{\Lambda}^{unr}. $$
\end{corollary}
\begin{proof}
As in \cref{finite_level}, we let
$$N_{n} = \omega_{n} {\rm Div}_{\Lambda}(X_{\infty}) + {\rm Pr}_{\Lambda,n}^{ram} + {\rm Pr}_{\Lambda}^{unr}$$
as long as $n \ge k = {\rm max}\{k_{i}: i = r+1,\ldots,s\}$.  Observe first that the commutativity of the diagram
\begin{equation*}
\begin{tikzcd}
{\rm Div}_{\Lambda}(X_{\infty}) \arrow[r,"\pi_{n+1}"] \arrow{dr}[swap]{\pi_{n}}  & {\rm Div}_{p}(X_{n+1}) \arrow[d] \\
& {\rm Div}_{p}(X_{n})
\end{tikzcd} 
\end{equation*}
induces yet another commutative diagram
\begin{equation*}
\begin{tikzcd}
{\rm Div}_{\Lambda}(X_{\infty})/N_{n+1} \arrow[r,"\simeq"] \arrow[d] & {\rm Pic}_{p}(X_{n+1}) \arrow[d] \\
{\rm Div}_{\Lambda}(X_{\infty})/N_{n} \arrow[r,"\simeq"] & {\rm Pic}_{p}(X_{n}),
\end{tikzcd}
\end{equation*}
where the left vertical arrow is the natural map induced from the inclusion $N_{n+1} \subseteq N_{n}$, and the two horizontal isomorphisms are given by \cref{finite_level}.  It follows that we have
$${\rm Pic}_{\Lambda} \simeq \varprojlim_{n \ge k} {\rm Div}_{\Lambda}(X_{\infty})/N_{n}. $$
Since ${\rm Pr}_{\Lambda}^{unr} \subseteq N_{n}$, we have natural maps
$${\rm Div}_{\Lambda}(X_{\infty})/{\rm Pr}_{\Lambda}^{unr} \rightarrow {\rm Div}_{\Lambda}(X_{\infty})/N_{n} $$
that induce a $\Lambda$-module morphism
\begin{equation} \label{concrete_iso}
{\rm Div}_{\Lambda}(X_{\infty})/{\rm Pr}_{\Lambda}^{unr} \rightarrow \varprojlim_{n \ge k} {\rm Div}_{\Lambda}(X_{\infty})/N_{n}, 
\end{equation}
and it remains to show that this last morphism is in fact an isomorphism.  Observe that 
$$\bigcap_{n \ge k}N_{n} = {\rm Pr}_{\Lambda}^{unr}. $$
Indeed, the inclusion ${\rm Pic}_{\Lambda}^{unr} \subseteq \bigcap_{n \ge k}N_{n}$ is clear by definition, and if $D \in \bigcap_{n \ge k}N_{n}$, then for each $n \ge k$, we have $D = D_{n}' + P_{n}$ for some $D_{n}' \in \omega_{n} \cdot {\rm Div}_{\Lambda}(X_{\infty}) + {\rm Pr}_{\Lambda,n}^{ram}$ and some $P_{n} \in {\rm Pr}_{\Lambda}^{unr}$.  Since $D_{n}' \to 0$ as $n \to \infty$, we have $P_{n} \to D$ as $n \to \infty$.  But since ${\rm Pr}_{\Lambda}^{unr}$ is finitely generated over $\Lambda$, it is compact.  It follows that $D \in {\rm Pr}_{\Lambda}^{unr}$, and this shows that (\ref{concrete_iso}) is injective.  The surjectivity of (\ref{concrete_iso}) follows from a standard result about projective limits of compact modules (see for instance \cite[Chapter IV, Proposition 2.7]{Neukirch:1999}).  This ends the proof that
$${\rm Pic}_{\Lambda} \simeq {\rm Div}_{\Lambda}(X_{\infty})/{\rm Pr}_{\Lambda}^{unr},$$
and the second isomorphism is proven similarly.

\end{proof}
Since ${\rm Div}_{\Lambda}(X_{\infty})$ is a finitely generated $\Lambda$-module, then so is ${\rm Pic}_{\Lambda}$ by \cref{exp_desc}.  Moreover, since $\Lambda$ is a noetherian ring, ${\rm Div}_{\Lambda}^{0}(X_{\infty}) = {\rm ker}(s_{\Lambda})$ is a finitely generated $\Lambda$-module as well, and it follows from \cref{exp_desc} that ${\rm Pic}_{\Lambda}^{0}$ is also a finitely generated $\Lambda$-module.

%\subsection{The analogue of Iwasawa's asymptotic class number formula}
\subsection{The analogue of Iwasawa's asymptotic class number formula} \label{asymptotic1}
We can now go ahead and prove our first main theorem (\cref{main1} from \cref{Introduction}) of this paper.
\begin{theorem} \label{analogue_iwasawa}
Let $X = (V_{X},\mathbf{E}_{X})$ be a finite connected graph, and let $\alpha:\mathbf{E}_{X} \rightarrow \Gamma$ be a voltage assignment.  Consider a family $\mathcal{I}$ of closed subgroups of $\Gamma$ indexed by the vertices of $V_{X}$ and consider the branched $\mathbb{Z}_{p}$-tower of graphs
$$X = X_{0} \leftarrow X_{1} \leftarrow X_{2} \leftarrow \ldots \leftarrow X_{n} \leftarrow \ldots,$$
where $X_{n} = X(\Gamma_{n},\mathcal{I}_{n},\alpha_{n})$ as explained in \cref{towers}.  Assume that all finite graphs $X_{n}$ are connected.  Then ${\rm Pic}_{\Lambda}^{0}$ is a finitely generated torsion $\Lambda$-module.  Moreover, if we let
$$\mu = \mu({\rm Pic}_{\Lambda}^{0}) \text{ and } \lambda = \lambda({\rm Pic}_{\Lambda}^{0}), $$
then there exist $n_{0}\in \mathbb{Z}_{\ge 0}$ and $\nu \in \mathbb{Z}$ such that
$${\rm ord}_{p}(\kappa(X_{n})) = \mu p^{n} + \lambda n + \nu, $$
when $n \ge n_{0}$.
\end{theorem}
\begin{proof}
We keep the same notation as above.  As in the proof of \cref{exp_desc}, set $k = {\rm max}\{k_{i}: i=r+1,\ldots,s \}$, and define the $\Lambda$-modules
$$M = {\rm Div}_{\Lambda}^{0}(X_{\infty})/{\rm Pr}_{\Lambda}^{unr} \text{ and } A = N/{\rm Pr}_{\Lambda}^{unr}, $$
where 
$$N = \omega_{k} {\rm Div}_{\Lambda}(X_{\infty}) + \frac{\omega_{k}}{\omega_{n}} {\rm Pr}_{\Lambda,n}^{ram} + {\rm Pr}_{\Lambda}^{unr}. $$
Note that $M \simeq {\rm Pic}_{\Lambda}^{0}$ by \cref{exp_desc}, and that $N$ does not depend on $n$.  \cref{finite_level} implies that 
$${\rm Pic}_{p}^{0}(X_{n}) \simeq M/ \omega_{n,k} A, $$
where the elements $\omega_{n,k} \in R$ were defined in \cref{Iwasawa_algebra}.  It follows that we have a short exact sequence
\begin{equation} \label{nice_ses}
0 \rightarrow A/\omega_{n,k}A \rightarrow {\rm Pic}_{p}^{0}(X_{n}) \rightarrow M/A \rightarrow 0
\end{equation}
of $\Lambda$-modules.  Since ${\rm Pic}_{p}^{0}(X_{n}) \simeq {\rm Pic}^{0}(X_{n})[p^{\infty}]$ is finite, then so are $A/\omega_{n,k}A$ and $M/A$.  By definition, $A$ is a finitely generated $\Lambda$-module.  It follows from \cref{washington} that $A$ is torsion and that there exist $n_{0} \in \mathbb{Z}_{\ge 0}$ and $\nu' \in \mathbb{Z}$ such that when $n \ge n_{0}$, one has
$${\rm ord}_{p}(|A/\omega_{n,k}A|) = \mu(A) p^{n} + \lambda(A) n + \nu'.$$
Since $M/A \simeq {\rm Div}_{\Lambda}^{0}(X_{\infty})/N$ does not depend on $n$, we have $|M/A| = p^{k'}$ for some $k' \ge 0$.  Letting $\nu = \nu' + k'$, the short exact sequence (\ref{nice_ses}) gives
$${\rm ord}_{p}(\kappa(X_{n})) = \mu(A) p^{n} + \lambda(A) n + \nu, $$
when $n$ is large.  The $\Lambda$-modules $M$ and $A$ are related to one another via the short exact sequence
$$0 \rightarrow A \rightarrow M \rightarrow M/A \rightarrow 0, $$
and since $M/A$ is finite, $A$ and $M$ are pseudo-isomorphic.  It follows that $M \simeq {\rm Pic}_{\Lambda}^{0}$ is also a finitely generated torsion $\Lambda$-module and
$$\mu(A) = \mu(M) \text{ and } \lambda(A) = \lambda(M). $$
\end{proof}
\begin{remark}
In the situation where $V^{ram} = \varnothing$, note that we can take $k=0$ and the short exact sequence (\ref{nice_ses}) becomes
$$0 \rightarrow A/\omega_{n,0}A \rightarrow {\rm Pic}_{p}^{0}(X_{n}) \rightarrow {\rm Pic}_{p}^{0}(X) \rightarrow 0$$
for all $n \ge 0$.  This corresponds to the situation originally studied from the module theoretical point of view in \cite{Gonet:2022}, \cite{Kleine/Muller:2023}, and \cite{Kataoka:2024}.
\end{remark}

%\subsection{The analogue of Iwasawa's main conjecture}
\subsection{The analogue of Iwasawa's main conjecture} \label{main_conj1}
Let us start with the following proposition.
\begin{proposition} \label{pic_vs_zero}
Let $X = (V_{X},\mathbf{E}_{X})$ be a finite connected graph, and let $\alpha:\mathbf{E}_{X} \rightarrow \Gamma$ be a voltage assignment.  Consider a family $\mathcal{I}$ of closed subgroups of $\Gamma$ indexed by the vertices of $V_{X}$ and consider the branched $\mathbb{Z}_{p}$-tower of graphs
$$X = X_{0} \leftarrow X_{1} \leftarrow X_{2} \leftarrow \ldots \leftarrow X_{n} \leftarrow \ldots,$$
where $X_{n} = X(\Gamma_{n},\mathcal{I}_{n},\alpha_{n})$ as explained in \cref{towers}.  Assume that all finite graphs $X_{n}$ are connected.  Then, one has a short exact sequence of Iwasawa modules
$$0 \rightarrow {\rm Pic}_{\Lambda}^{0} \rightarrow {\rm Pic}_{\Lambda} \rightarrow \mathbb{Z}_{p} \rightarrow 0, $$
where $\mathbb{Z}_{p}$ is viewed as a $\Lambda$-module with trivial action.  Moreover, ${\rm Pic}_{\Lambda}$ is a torsion Iwasawa module, and one has
$${\rm char}_{\Lambda}({\rm Pic}_{\Lambda}^{0}) \cdot \omega_{0}\Lambda = {\rm char}_{\Lambda}({\rm Pic}_{\Lambda}). $$
\end{proposition}
\begin{proof}
The short exact sequence (\ref{useful_ses}) induces a short exact sequence
$$0 \longrightarrow {\rm Div}_{\Lambda}^{0}/{\rm Pr}_{\Lambda}^{unr} \longrightarrow {\rm Div}_{\Lambda}/{\rm Pr}_{\Lambda}^{unr} \stackrel{s_{\Lambda}}{\longrightarrow} \mathbb{Z}_{p} \longrightarrow 0, $$
which combined with \cref{exp_desc} gives the desired short exact sequence.  Since both $\mathbb{Z}_{p}$ and ${\rm Pic}_{\Lambda}^{0}$ are torsion Iwasawa modules, it follows that ${\rm Pic}_{\Lambda}$ is also a torsion Iwasawa module.  Moreover, since ${\rm char}_{\Lambda}(\mathbb{Z}_{p}) = \omega_{0}\Lambda$, (\ref{ideal_mult}) implies that we have the equality of ideals
$${\rm char}_{\Lambda}({\rm Pic}_{\Lambda}^{0}) \cdot \omega_{0}\Lambda = {\rm char}_{\Lambda}({\rm Pic}_{\Lambda}). $$
\end{proof}
It follows from \cref{pic_vs_zero} that in order to understand ${\rm char}_{\Lambda}({\rm Pic}_{\Lambda}^{0})$, it suffices to understand ${\rm char}_{\Lambda}({\rm Pic}_{\Lambda})$.  In order to do so, consider $X_{\infty}^{unr} = X(\Gamma,\alpha)$ as we did at the end of \cref{connectedness}.  Let 
$$M_{1} = \bigoplus_{i=1}^{r} \Lambda \cdot w_{i,\infty}^{unr} \text{ and } M_{2} = \bigoplus_{i=r+1}^{s} \Lambda \cdot w_{i,\infty}^{unr},$$
where $w_{i,\infty}^{unr} = (v_{i},1_{\Gamma})$ for $i=1,\ldots,s$.  Note that we have ${\rm Div}_{\Lambda}(X_{\infty}^{unr}) = M_{1} \oplus M_{2}$.  We define an operator
\begin{equation} \label{oper_delta}
\Delta:{\rm Div}_{\Lambda}(X_{\infty}^{unr}) \rightarrow {\rm Div}_{\Lambda}(X_{\infty}^{unr}) 
\end{equation}
as follows:  If $i=1,\ldots,r$, then 
$$\Delta(w_{i,\infty}^{unr}) = \mathcal{L}^{unr}_{\Lambda}(w_{i,\infty}^{unr}),$$ 
where $\mathcal{L}^{unr}_{\Lambda}$ is the Laplacian operator on the locally finite graph $X_{\infty}^{unr}$, and if $i=r+1,\ldots,s$, then
$$\Delta(w_{i,\infty}^{unr}) = \omega_{k_{i}} \cdot w_{i,\infty}^{unr}. $$
Since ${\rm Div}_{\Lambda}(X_{\infty}^{unr})$ is a free $\Lambda$-module, it makes sense to talk about ${\rm det}(\Delta) \in \Lambda$.  We are now ready to prove our second main theorem (\cref{main3} from \cref{Introduction}) of this paper.  (See also \cref{main_conj_power_sr} below.)
\begin{theorem} \label{main_conj}
Let $X = (V_{X},\mathbf{E}_{X})$ be a finite connected graph, and let $\alpha:\mathbf{E}_{X} \rightarrow \Gamma$ be a voltage assignment.  Consider a family $\mathcal{I}$ of closed subgroups of $\Gamma$ indexed by the vertices of $V_{X}$ and consider the branched $\mathbb{Z}_{p}$-tower of graphs
$$X = X_{0} \leftarrow X_{1} \leftarrow X_{2} \leftarrow \ldots \leftarrow X_{n} \leftarrow \ldots,$$
where $X_{n} = X(\Gamma_{n},\mathcal{I}_{n},\alpha_{n})$ as explained in \cref{towers}.  Assume that all finite graphs $X_{n}$ are connected.  With the notation as above, one has
$${\rm char}_{\Lambda}({\rm Pic}_{\Lambda}) = ({\rm det}(\Delta)), $$
where $\Delta$ is the operator defined above in (\ref{oper_delta}).  Moreover, one has
$${\rm char}_{\Lambda}({\rm Pic}_{\Lambda}^{0}) \cdot \omega_{0}\Lambda = ({\rm det}(\Delta)). $$
\end{theorem}
\begin{proof}
Consider the natural surjective morphism 
$\iota_{*}:{\rm Div}_{\Lambda}(X_{\infty}^{unr}) \rightarrow {\rm Div}_{\Lambda}(X_{\infty})$
of $\Lambda$-modules induced by the immersion $\iota: X_{\infty}^{unr} \rightarrow X_{\infty}$.  It follows from the explicit description (\ref{struc_mod_inf}) of the $\Lambda$-module structure of ${\rm Div}_{\Lambda}(X_{\infty})$ that 
$${\rm ker}(\iota_{*}) = \bigoplus_{i=r+1}^{s} \omega_{k_{i}}\Lambda \cdot w_{i,\infty}^{unr}. $$
Therefore, we have an exact sequence of $\Lambda$-modules
$${\rm Div}_{\Lambda}(X_{\infty}^{unr}) \stackrel{\Delta}{\longrightarrow} {\rm Div}_{\Lambda}(X_{\infty}^{unr}) \longrightarrow {\rm Div}_{\Lambda}(X_{\infty})/{\rm Pr}_{\Lambda}^{unr} \longrightarrow 0, $$
where the map on the right is $\iota_{*}$ followed by the natural projection map ${\rm Div}_{\Lambda}(X_{\infty}) \rightarrow {\rm Div}_{\Lambda}(X_{\infty})/{\rm Pr}_{\Lambda}^{unr}$.  By \cref{exp_desc}, we have ${\rm Div}_{\Lambda}(X_{\infty})/{\rm Pr}_{\Lambda}^{unr} \simeq {\rm Pic}_{\Lambda}$.  Moreover, since ${\rm Pic}_{\Lambda}$ is a torsion $\Lambda$-module by \cref{pic_vs_zero} and ${\rm Div}_{\Lambda}(X_{\infty}^{unr})$ is free of finite rank as a $\Lambda$-module, \cite[Lemma A.3]{Kataoka:2024} implies that $\Delta$ is injective so that we have a short exact sequence
$$0 \longrightarrow {\rm Div}_{\Lambda}(X_{\infty}^{unr}) \stackrel{\Delta}{\longrightarrow} {\rm Div}_{\Lambda}(X_{\infty}^{unr}) \longrightarrow {\rm Pic}_{\Lambda} \longrightarrow 0. $$
By \cref{char_ideal_calc}, we have
$${\rm char}_{\Lambda}({\rm Pic}_{\Lambda}) = ({\rm det}(\Delta)). $$
The second equality follows directly from \cref{pic_vs_zero}.
\end{proof}
\begin{remark}
When $V^{ram} = \varnothing$, \cref{main_conj} reduces to \cite[Theorem 5.2]{Kleine/Muller:2023} in the situation where $l=1$ in their notation.  In this case, we have $X_{n} = X_{n}^{unr}$ for all $n \ge 0$, and $\Delta$ is the Laplacian operator $\mathcal{L}_{\Lambda}^{unr}$ on the locally finite graph $X_{\infty} = X_{\infty}^{unr}$.
\end{remark}

%\section{Examples}
\section{Examples} \label{examples}
In order to provide numerical examples, it is convenient to work with $\mathbb{Z}_{p}\llbracket T \rrbracket$ rather than $\Lambda$.  So we now reexpress the result of \cref{main_conj} using the non-canonical isomorphism (\ref{non_can}).  Moreover, we prefer to work additively, so we work directly with $\mathbb{Z}_{p}$ rather than $\Gamma$.  We will make use of the group morphism
\begin{equation} \label{useful_gr_mor}
\rho: \mathbb{Z}_{p} \rightarrow \mathbb{Z}_{p}\llbracket T\rrbracket^{\times} 
\end{equation}
given by 
$$a \mapsto \rho(a) = (1 + T)^{a} = \sum_{i=0}^{\infty}\binom{a}{i} T^{i} \in \mathbb{Z}_{p}\llbracket T \rrbracket ^{\times},$$
where the $\binom{X}{i}$ are the usual integer-valued polynomials given by $\binom{X}{0} = 1$, $\binom{X}{1} = X$, and
$$\binom{X}{i} = \frac{X(X-1)\cdots (X - i + 1)}{i!} \in \mathbb{Q}[X] $$
for $i \ge 2$.  Note that the group morphism $\rho$ corresponds to $\Gamma \hookrightarrow R^{\times} \hookrightarrow \Lambda^{\times}$.  Starting with a voltage assignment $\alpha: \mathbf{E}_{X} \rightarrow \mathbb{Z}_{p}$, we define a few matrices as follows.  First, we let $D = (d_{ij})$ be the diagonal matrix defined via
\begin{equation*}
d_{ii} = 
\begin{cases}
{\rm val}_{X}(v_{i}), &\text{ if } i=1,\ldots,r;\\
0, & \text{ otherwise. }
\end{cases}
\end{equation*}
Then, we define another matrix 
$$B(T) = (b_{ij}(T)) \in M_{s \times s}(\mathbb{Z}_{p}\llbracket T \rrbracket)$$
as follows:
\begin{enumerate}
\item If $i=1,\ldots,s$, and $j=1,\ldots,r$, then let
$$b_{ij}(T) = \sum_{\substack{e \in \mathbf{E}_{X} \\ {\rm inc}(e) = (v_{j},v_{i})}} \rho(\alpha(e)), $$
where $\rho$ is the group morphism (\ref{useful_gr_mor}) above.
\item If $i = j = r+1,\ldots,s$, then let
$$b_{ij}(T) = -\omega_{k_{i}}(T), $$
where $\omega_{k_{i}}(T) = (1+T)^{p^{k_{i}}} - 1$ is the distinguished polynomial in $\mathbb{Z}_{p}\llbracket T \rrbracket$ corresponding to $\omega_{k_{i}}$ via the isomorphism (\ref{non_can}). 
\item Set $b_{ij}(T) = 0$ otherwise.  
\end{enumerate}
A simple calculation based on the definition (\ref{oper_delta}) of the operator $\Delta$ shows that 
$${\rm det}(\Delta) = {\rm det}(D - B(T)), $$
when $\Delta$ is viewed as a morphism of $\mathbb{Z}_{p}\llbracket T \rrbracket$-modules.  We let
\begin{equation} \label{power_ser}
f_{X,\mathcal{I},\alpha}(T) = {\rm det}(D - B(T)) \in \mathbb{Z}_{p}\llbracket T \rrbracket.
\end{equation}
\cref{main_conj} can now be phrased as follows.
\begin{theorem} \label{main_conj_power_sr}
Let $X = (V_{X},\mathbf{E}_{X})$ be a finite connected graph, and let $\alpha:\mathbf{E}_{X} \rightarrow \mathbb{Z}_{p}$ be a voltage assignment.  Consider a family $\mathcal{I}$ of closed subgroups of $\mathbb{Z}_{p}$ indexed by the vertices of $V_{X}$ and consider the branched $\mathbb{Z}_{p}$-tower of graphs
$$X = X_{0} \leftarrow X_{1} \leftarrow X_{2} \leftarrow \ldots \leftarrow X_{n} \leftarrow \ldots,$$
where $X_{n} = X(\mathbb{Z}/p^{n}\mathbb{Z},\mathcal{I}_{n},\alpha_{n})$ as explained in \cref{towers}.  Assume that all finite graphs $X_{n}$ are connected.  Then, one has
$${\rm char}_{\mathbb{Z}_{p}\llbracket T \rrbracket}({\rm Pic}_{\Lambda}^{0}) \cdot (T) = (f_{X,\mathcal{I},\alpha}(T)), $$
where $f_{X,\mathcal{I},\alpha}(T)$ is the power series defined in (\ref{power_ser}) above.
\end{theorem}
It follows from \cref{main_conj_power_sr} that we have
$$\mu({\rm Pic}_{\Lambda}^{0}) = \mu(f_{X,\mathcal{I},\alpha}(T)) \text{ and } \lambda({\rm Pic}_{\Lambda}^{0}) = \lambda(f_{X,\mathcal{I},\alpha}(T)) - 1, $$
and this allows us to provide a few numerical examples below.  The calculation of the number of spanning trees and of the power series $f_{X,\mathcal{I},\alpha}(T)$ below have been performed with the sofware \cite{SAGE}.

\begin{enumerate}
\item Let us revisit Example $2$ on page $451$ of \cite{Vallieres:2021}.  Let $p=2$, and let $X$ be the bouquet graph on two loops.  Pick an arbitrary orientation $S$ for $X$ and consider the function $\alpha:S \rightarrow \mathbb{Z} \subseteq \mathbb{Z}_{2}$ given by
$$s_{1} \mapsto \alpha(s_{1}) = 3, \text{ and } s_{3} \mapsto \alpha(s_{3}) = 5. $$
The corresponding (unramified) $\mathbb{Z}_{2}$-tower of graphs was studied in \cite[page 451]{Vallieres:2021}.  One can use \cref{connect} to check that all the finite graphs in the $\mathbb{Z}_{2}$-tower of graphs are connected.  Let us now introduce some ramification at the unique vertex of $X$.  Say, we let $I_{1} = 4\mathbb{Z}_{2}$.  Then, we obtain the branched $\mathbb{Z}_{2}$-tower of graphs

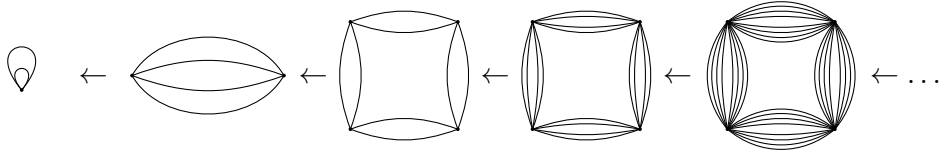
\begin{equation*}
\begin{tikzpicture}[baseline={([yshift=-1.7ex] current bounding box.center)}]
% create the node
\node[draw=none,minimum size=2cm,regular polygon,regular polygon sides=1] (a) {};
% draw a black dot in each vertex
\foreach \x in {1}
  \fill (a.corner \x) circle[radius=0.7pt];
\draw (a.corner 1) to [in=50,out=130,loop] (a.corner 1);
\draw (a.corner 1) to [in=50,out=130,distance = 0.5cm,loop] (a.corner 1);
\end{tikzpicture}
\leftarrow \, \, \,
\begin{tikzpicture}[baseline={([yshift=-0.6ex] current bounding box.center)}]
% create the node
\node[draw=none,minimum size=2cm,regular polygon,rotate = -45,regular polygon sides=4] (a) {};

% draw a black dot in each vertex
  \fill (a.corner 1) circle[radius=0.7pt];
  \fill (a.corner 3) circle[radius=0.7pt];
  
  \path (a.corner 1) edge [bend left=20] (a.corner 3);
  \path (a.corner 1) edge [bend left=60] (a.corner 3);
  \path (a.corner 1) edge [bend right=20] (a.corner 3);
  \path (a.corner 1) edge [bend right=60] (a.corner 3);

\end{tikzpicture}
 \, \leftarrow 
\begin{tikzpicture}[baseline={([yshift=-0.6ex] current bounding box.center)}]
% create the node
\node[draw=none,minimum size=2cm,regular polygon,regular polygon sides=4] (a) {};

% draw a black dot in each vertex
\foreach \x in {1,2,...,4}
\fill (a.corner \x) circle[radius=0.7pt];

\path (a.corner 1) edge [bend left=20] (a.corner 2);
\path (a.corner 1) edge [bend right=20] (a.corner 2);
\path (a.corner 2) edge [bend left=20] (a.corner 3);
\path (a.corner 2) edge [bend right=20] (a.corner 3);
\path (a.corner 3) edge [bend left=20] (a.corner 4);
\path (a.corner 3) edge [bend right=20] (a.corner 4);
\path (a.corner 4) edge [bend left=20] (a.corner 1);
\path (a.corner 4) edge [bend right=20] (a.corner 1);
\end{tikzpicture}
\leftarrow
\begin{tikzpicture}[baseline={([yshift=-0.6ex] current bounding box.center)}]
% create the node
\node[draw=none,minimum size=2cm,regular polygon,regular polygon sides=4] (a) {};
    
% draw a black dot in each vertex
\foreach \x in {1,2,...,4}
\fill (a.corner \x) circle[radius=0.7pt];
    
\path (a.corner 1) edge [bend left=20] (a.corner 2);
\path (a.corner 1) edge [bend right=20] (a.corner 2);
\path (a.corner 1) edge [bend left=10] (a.corner 2);
\path (a.corner 1) edge [bend right=10] (a.corner 2);

\path (a.corner 2) edge [bend left=20] (a.corner 3);
\path (a.corner 2) edge [bend right=20] (a.corner 3);
\path (a.corner 2) edge [bend left=10] (a.corner 3);
\path (a.corner 2) edge [bend right=10] (a.corner 3);

\path (a.corner 3) edge [bend left=20] (a.corner 4);
\path (a.corner 3) edge [bend right=20] (a.corner 4);
\path (a.corner 3) edge [bend left=10] (a.corner 4);
\path (a.corner 3) edge [bend right=10] (a.corner 4);

\path (a.corner 4) edge [bend left=20] (a.corner 1);
\path (a.corner 4) edge [bend right=20] (a.corner 1);
\path (a.corner 4) edge [bend left=10] (a.corner 1);
\path (a.corner 4) edge [bend right=10] (a.corner 1);
\end{tikzpicture}
\leftarrow 
\begin{tikzpicture}[baseline={([yshift=-0.6ex] current bounding box.center)}]
% create the node
\node[draw=none,minimum size=2cm,regular polygon,regular polygon sides=4] (a) {};
    
% draw a black dot in each vertex
\foreach \x in {1,2,...,4}
\fill (a.corner \x) circle[radius=0.7pt];
    
\path (a.corner 1) edge [bend left=20] (a.corner 2);
\path (a.corner 1) edge [bend right=20] (a.corner 2);
\path (a.corner 1) edge [bend left=10] (a.corner 2);
\path (a.corner 1) edge [bend right=10] (a.corner 2);
\path (a.corner 1) edge [bend left=30] (a.corner 2);
\path (a.corner 1) edge [bend right=30] (a.corner 2);
\path (a.corner 1) edge [bend left=40] (a.corner 2);
\path (a.corner 1) edge [bend right=40] (a.corner 2);

\path (a.corner 2) edge [bend left=20] (a.corner 3);
\path (a.corner 2) edge [bend right=20] (a.corner 3);
\path (a.corner 2) edge [bend left=10] (a.corner 3);
\path (a.corner 2) edge [bend right=10] (a.corner 3);
\path (a.corner 2) edge [bend left=30] (a.corner 3);
\path (a.corner 2) edge [bend right=30] (a.corner 3);
\path (a.corner 2) edge [bend left=40] (a.corner 3);
\path (a.corner 2) edge [bend right=40] (a.corner 3);

\path (a.corner 3) edge [bend left=20] (a.corner 4);
\path (a.corner 3) edge [bend right=20] (a.corner 4);
\path (a.corner 3) edge [bend left=10] (a.corner 4);
\path (a.corner 3) edge [bend right=10] (a.corner 4);
\path (a.corner 3) edge [bend left=30] (a.corner 4);
\path (a.corner 3) edge [bend right=30] (a.corner 4);
\path (a.corner 3) edge [bend left=40] (a.corner 4);
\path (a.corner 3) edge [bend right=40] (a.corner 4);

\path (a.corner 4) edge [bend left=20] (a.corner 1);
\path (a.corner 4) edge [bend right=20] (a.corner 1);
\path (a.corner 4) edge [bend left=10] (a.corner 1);
\path (a.corner 4) edge [bend right=10] (a.corner 1);
\path (a.corner 4) edge [bend left=30] (a.corner 1);
\path (a.corner 4) edge [bend right=30] (a.corner 1);
\path (a.corner 4) edge [bend left=40] (a.corner 1);
\path (a.corner 4) edge [bend right=40] (a.corner 1);
\end{tikzpicture}
\leftarrow \ldots
\end{equation*}
All the graphs $X_{n}$ are connected by \cref{conn_unr}.  The power series $f_{X,\mathcal{I},\alpha}(T)$ is given by
$$f_{X,\mathcal{I},\alpha}(T) =  4T + 6T^{2} + 4T^{3} + T^{4} \in \mathbb{Z}[T] \subseteq \mathbb{Z}_{2}\llbracket T \rrbracket,$$
so we should have $\mu({\rm Pic}_{\Lambda}^{0}) = 0$ and $\lambda({\rm Pic}_{\Lambda}^{0}) = 3$.  We calculate
$$\kappa_{0} = 1, \kappa_{1} = 2^{2},\kappa_{2} = 2^{5}, \kappa_{3} = 2^{8}, \ldots,$$
where we write $\kappa_{n}$ for $\kappa(X_{n})$, but it is simple to see combinatorially that $\kappa_{n} = 2^{3n-1}$.  We have
$${\rm ord}_{2}(\kappa(X_{n})) = 3n-1, $$
for all $n \ge 1$.

\item Let us revisit Example $2$ on page $55$ of \cite{DLRV}.  Let $p=3$ and let $X$ be the dumbbell graph
\begin{center}
\begin{tikzpicture}

%vertices
\draw[fill=black] (0,0) circle (0.7pt);
\draw[fill=black] (0.6,0) circle (0.7pt);        
        
%edges
\draw (0,0) edge (0.6,0);
\draw (0,0) edge [loop left, in = 155, out = 205,min distance=4mm] (0,0) ;
\draw (0.6,0) edge [loop right, in = 25, out = 335,min distance=4mm] (0.6,0) ;
\end{tikzpicture}
\end{center}
Pick the orientation $S$ to be arbitary on the loops and going from the first vertex on the left to the vertex on the right for the non-loop edge.  Label those directed edges from left to right so that $s_{2}$ corresponds to the non-loop directed edge going from left to right.  Take the function $\alpha:S \rightarrow \mathbb{Z} \subseteq \mathbb{Z}_{3}$ given by
$$s_{1} \mapsto \alpha(s_{1}) = 1, s_{2} \mapsto \alpha(s_{2}) =0, \text{ and } s_{3} \mapsto \alpha(s_{3}) = 11. $$  
One obtains the following (unramified) $\mathbb{Z}_{3}$-tower of finite graphs
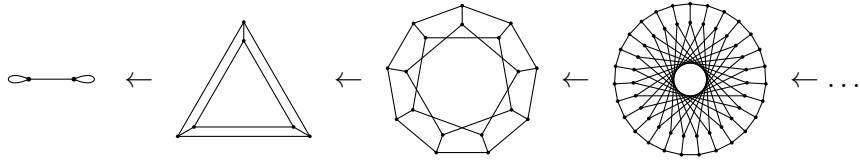
\begin{equation*}
\begin{tikzpicture}[baseline={([yshift=-0.6ex] current bounding box.center)}]

%vertices
\draw[fill=black] (0,0) circle (0.7pt);
\draw[fill=black] (0.6,0) circle (0.7pt);        
                
%edges
\draw (0,0) edge  (0.6,0);
\draw (0,0) edge [loop left, in = 155, out = 205,min distance=4mm] (0,0) ;
\draw (0.6,0) edge [loop right, in = 25, out = 335,min distance=4mm] (0.6,0) ;
\end{tikzpicture}
\, \, \, \leftarrow \, \, \,
\begin{tikzpicture}[baseline={([yshift=-0.6ex] current bounding box.center)}]
% create the node
\node[draw=none,minimum size=2cm,regular polygon,regular polygon sides=3] (a) {};
\node[draw=none, minimum size=1.5cm,regular polygon,regular polygon sides=3] (b) {};
    
% draw a black dot in each vertex
\foreach \x in {1,2,...,3}
\fill (a.corner \x) circle[radius=0.7pt];
      
\foreach \y in {1,2,...,3}
\fill (b.corner \y) circle[radius=0.7pt];
      
\path (a.corner 1) edge (a.corner 2);
\path (a.corner 2) edge (a.corner 3);
\path (a.corner 3) edge (a.corner 1);

\path (a.corner 1) edge (b.corner 1);
\path (a.corner 2) edge (b.corner 2);
\path (a.corner 3) edge (b.corner 3);

\path (b.corner 1) edge (b.corner 3);
\path (b.corner 2) edge (b.corner 1);
\path (b.corner 3) edge (b.corner 2);

\end{tikzpicture}
\, \, \, \leftarrow \, \, \,
\begin{tikzpicture}[baseline={([yshift=-0.6ex] current bounding box.center)}]
% create the node
\node[draw=none,minimum size=2cm,regular polygon,regular polygon sides=9] (a) {};
\node[draw=none, minimum size=1.5cm,regular polygon,regular polygon sides=9] (b) {};
 
% draw a black dot in each vertex
\foreach \x in {1,2,...,9}
\fill (a.corner \x) circle[radius=0.7pt];
      
\foreach \y in {1,2,...,9}
\fill (b.corner \y) circle[radius=0.7pt];
      
\foreach \x\z in {1/2,2/3,3/4,4/5,5/6,6/7,7/8,8/9,9/1}
\path (a.corner \x) edge (a.corner \z);  
     
\foreach \x\z in {1/1,2/2,3/3,4/4,5/5,6/6,7/7,8/8,9/9} 
\path (a.corner \x) edge (b.corner \z);
      
\foreach \x\z in {1/3,2/4,3/5,4/6,5/7,6/8,7/9,8/1,9/2} 
\path (b.corner \x) edge (b.corner \z);
      
\end{tikzpicture}
\, \, \, \leftarrow \, \, \,
\begin{tikzpicture}[baseline={([yshift=-0.6ex] current bounding box.center)}]
% create the node
\node[draw=none,minimum size=2cm,regular polygon,regular polygon sides=27] (a) {};
\node[draw=none, minimum size=1.5cm,regular polygon,regular polygon sides=27] (b) {};
    
% draw a black dot in each vertex
\foreach \x in {1,2,...,27}
\fill (a.corner \x) circle[radius=0.7pt];
      
\foreach \y in {1,2,...,27}
\fill (b.corner \y) circle[radius=0.7pt];
      
\foreach \x\z in {1/2,2/3,3/4,4/5,5/6,6/7,7/8,8/9,9/10,10/11,11/12,12/13,13/14,14/15,15/16,16/17,17/18,18/19,19/20,20/21,21/22,22/23,23/24,24/25,25/26,26/27,27/1}
\path (a.corner \x) edge (a.corner \z);  
     
\foreach \x\z in {1/1,2/2,3/3,4/4,5/5,6/6,7/7,8/8,9/9,10/10,11/11,12/12,13/13,14/14,15/15,16/16,17/17,18/18,19/19,20/20,21/21,22/22,23/23,24/24,25/25,26/26,27/27} 
\path (a.corner \x) edge (b.corner \z);
      
\foreach \x\z in {1/12,2/13,3/14,4/15,5/16,6/17,7/18,8/19,9/20,10/21,11/22,12/23,13/24,14/25,15/26,16/27,17/1,18/2,19/3,20/4,21/5,22/6,23/7,24/8,25/9,26/10,27/11}
\path (b.corner \x) edge (b.corner \z);
      
\end{tikzpicture}
\, \, \, \leftarrow \ldots
\end{equation*}
The graphs $X_{n}$ are the generalized Petersen graphs $G(3^{n},11)$ and are connected by \cref{connect}.  The power series $f_{X,\alpha}(T)$ is given by
$$f_{X,\alpha}(T) =  -122T^{2} + 122T^{3} -1211T^{4} + \ldots \in \mathbb{Z}\llbracket T \rrbracket \subseteq \mathbb{Z}_{3}\llbracket T \rrbracket$$
so we should have $\mu({\rm Pic}_{\Lambda}^{0}) = 0$ and $\lambda({\rm Pic}_{\Lambda}^{0}) = 1$.  We calculate
$$\kappa_{0} = 1, \kappa_{1} = 3 \cdot 5^{2} , \kappa_{2} = 3^{2} \cdot 5^{2} \cdot 71^{2}, \kappa_{3} = 3^{3} \cdot 5^{2} \cdot 71^{2} \cdot 109^{2} \cdot 13931^{2}, \ldots,$$
where we write $\kappa_{n}$ for $\kappa(X_{n})$.  We have
$${\rm ord}_{3}(\kappa(X_{n})) = n, $$
for all $n \ge 0$.  Let us now introduce some ramification at the second vertex only, so let $I_{1} = 0$ and $I_{2} = 3\mathbb{Z}_{3}$.  Then, we obtain the branched $\mathbb{Z}_{3}$-tower of graphs
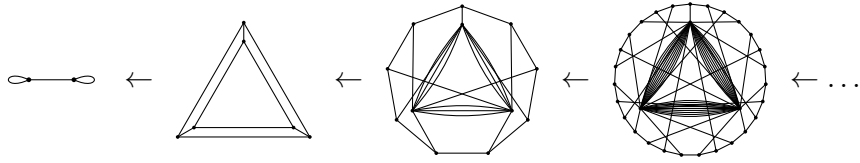
\begin{equation*}
\begin{tikzpicture}[baseline={([yshift=-0.6ex] current bounding box.center)}]

%vertices
\draw[fill=black] (0,0) circle (0.7pt);
\draw[fill=black] (0.6,0) circle (0.7pt);        
                
%edges
\draw (0,0) edge (0.6,0);
\draw (0,0) edge [loop left, in = 155, out = 205,min distance=4mm] (0,0) ;
\draw (0.6,0) edge [loop right, in = 25, out = 335,min distance=4mm] (0.6,0) ;
\end{tikzpicture}
\, \, \, \leftarrow \, \, \,
\begin{tikzpicture}[baseline={([yshift=-0.6ex] current bounding box.center)}]
% create the node
\node[draw=none,minimum size=2cm,regular polygon,regular polygon sides=3] (a) {};
\node[draw=none, minimum size=1.5cm,regular polygon,regular polygon sides=3] (b) {};

% draw a black dot in each vertex
\foreach \x in {1,2,...,3}
\fill (a.corner \x) circle[radius=0.7pt];
  
\foreach \y in {1,2,...,3}
\fill (b.corner \y) circle[radius=0.7pt];
  
\path (a.corner 1) edge (a.corner 2);
\path (a.corner 2) edge (a.corner 3);
\path (a.corner 3) edge (a.corner 1);

\path (a.corner 1) edge (b.corner 1);
\path (a.corner 2) edge (b.corner 2);
\path (a.corner 3) edge (b.corner 3);

\path (b.corner 1) edge (b.corner 3);
\path (b.corner 2) edge (b.corner 1);
\path (b.corner 3) edge (b.corner 2);

\end{tikzpicture}
\, \, \, \leftarrow \, \, \,
\begin{tikzpicture}[baseline={([yshift=-0.6ex] current bounding box.center)}]
% create the node
\node[draw=none,minimum size=2cm,regular polygon,regular polygon sides=9] (a) {};
\node[draw=none, minimum size=1.5cm,regular polygon,regular polygon sides=3] (b) {};
 
% draw a black dot in each vertex
\foreach \x in {1,2,...,9}
\fill (a.corner \x) circle[radius=0.7pt];
  
\foreach \y in {1,2,3}
\fill (b.corner \y) circle[radius=0.7pt];
  
\foreach \x\z in {1/2,2/3,3/4,4/5,5/6,6/7,7/8,8/9,9/1}
\path (a.corner \x) edge (a.corner \z);  
 
\foreach \x\z in {1/1,2/2,3/3,4/1,5/2,6/3,7/1,8/2,9/3} 
\path (a.corner \x) edge (b.corner \z);

\path (b.corner 1) edge [bend left=10] (b.corner 3);
\path (b.corner 1) edge [bend right=10] (b.corner 3);
\path (b.corner 1) edge  (b.corner 3);

\path (b.corner 2) edge [bend left=10] (b.corner 1);
\path (b.corner 2) edge [bend right=10] (b.corner 1);
\path (b.corner 2) edge  (b.corner 1);

\path (b.corner 3) edge [bend left=10] (b.corner 2);
\path (b.corner 3) edge [bend right=10] (b.corner 2);
\path (b.corner 2) edge  (b.corner 3);
  
\end{tikzpicture}
\, \, \, \leftarrow \, \, \,
\begin{tikzpicture}[baseline={([yshift=-0.6ex] current bounding box.center)}]
% create the node
\node[draw=none,minimum size=2cm,regular polygon,regular polygon sides=27] (a) {};
\node[draw=none, minimum size=1.5cm,regular polygon,regular polygon sides=3] (b) {};

% draw a black dot in each vertex
\foreach \x in {1,2,...,27}
\fill (a.corner \x) circle[radius=0.7pt];
  
\foreach \y in {1,2,3}
\fill (b.corner \y) circle[radius=0.7pt];
  
\foreach \x\z in {1/2,2/3,3/4,4/5,5/6,6/7,7/8,8/9,9/10,10/11,11/12,12/13,13/14,14/15,15/16,16/17,17/18,18/19,19/20,20/21,21/22,22/23,23/24,24/25,25/26,26/27,27/1}
\path (a.corner \x) edge (a.corner \z);  
 
\foreach \x\z in {1/1,2/2,3/3,4/1,5/2,6/3,7/1,8/2,9/3,10/1,11/2,12/3,13/1,14/2,15/3,16/1,17/2,18/3,19/1,20/2,21/3,22/1,23/2,24/3,25/1,26/2,27/3} 
\path (a.corner \x) edge (b.corner \z);

\path (b.corner 1) edge [bend left=4] (b.corner 3);
\path (b.corner 1) edge [bend left=8] (b.corner 3);
\path (b.corner 1) edge [bend left=12] (b.corner 3);
\path (b.corner 1) edge [bend left=16] (b.corner 3);

\path (b.corner 1) edge  (b.corner 3);

\path (b.corner 1) edge [bend right=4] (b.corner 3);
\path (b.corner 1) edge [bend right=8] (b.corner 3);
\path (b.corner 1) edge [bend right=12] (b.corner 3);
\path (b.corner 1) edge [bend right=16] (b.corner 3);

\path (b.corner 1) edge [bend left=4] (b.corner 2);
\path (b.corner 1) edge [bend left=8] (b.corner 2);
\path (b.corner 1) edge [bend left=12] (b.corner 2);
\path (b.corner 1) edge [bend left=16] (b.corner 2);

\path (b.corner 1) edge  (b.corner 2);

\path (b.corner 1) edge [bend right=4] (b.corner 2);
\path (b.corner 1) edge [bend right=8] (b.corner 2);
\path (b.corner 1) edge [bend right=12] (b.corner 2);
\path (b.corner 1) edge [bend right=16] (b.corner 2);

\path (b.corner 3) edge [bend left=4] (b.corner 2);
\path (b.corner 3) edge [bend left=8] (b.corner 2);
\path (b.corner 3) edge [bend left=12] (b.corner 2);
\path (b.corner 3) edge [bend left=16] (b.corner 2);

\path (b.corner 3) edge  (b.corner 2);

\path (b.corner 3) edge [bend right=4] (b.corner 2);
\path (b.corner 3) edge [bend right=8] (b.corner 2);
\path (b.corner 3) edge [bend right=12] (b.corner 2);
\path (b.corner 3) edge [bend right=16] (b.corner 2);
  
\end{tikzpicture}
\, \, \, \leftarrow \ldots
\end{equation*}
The graphs $X_{n}$ are connected by \cref{conn_unr}.  The power series $f_{X,\mathcal{I},\alpha}(T)$ is given by
$$f_{X,\mathcal{I},\alpha}(T) =  3T + 3T^{2} - 2T^{3} - \ldots \in \mathbb{Z}\llbracket T \rrbracket \subseteq \mathbb{Z}_{3}\llbracket T \rrbracket,$$
so we should have $\mu({\rm Pic}_{\Lambda}^{0}) = 0$ and $\lambda({\rm Pic}_{\Lambda}^{0}) = 3-1 = 2$.  We calculate
$$\kappa_{0} = 1, \kappa_{1} = 3 \cdot 5^{2},\kappa_{2} = 3^{3} \cdot 5^{2} \cdot 19^{2}, \kappa_{3} = 3^{5} \cdot 5^{2} \cdot 19^{2} \cdot 5779^{2},$$
$$\kappa_{4} = 3^{7} \cdot 5^{2} \cdot 19^{2} \cdot 3079^{2} \cdot 5779^{2} \cdot 62650261^{2}, \ldots,$$
where we let $\kappa_{n}$ be $\kappa(X_{n})$.  We have
$${\rm ord}_{3}(\kappa(X_{n})) = 2n-1, $$
for all $n \ge 1$.

\item Let $p=3$, and let $X$ be the graph
\begin{center}
\begin{tikzpicture}

%vertices
\draw[fill=black] (0,0) circle (0.7pt);
\draw[fill=black] (0.6,0) circle (0.7pt);        
                           
%edges
\draw (0,0) edge [bend left=20] (0.6,0);
\draw (0,0) edge (0.6,0);
\draw (0,0) edge [bend right=20] (0.6,0);
\draw (0,0) edge [loop left, in = 155, out = 205,min distance=4mm] (0,0) ;
\draw (0,0) edge [loop left, in = 145, out = 215,min distance=5mm] (0,0) ;
\draw (0,0) edge [loop left, in = 135, out = 225,min distance=6.5mm] (0,0) ;
    
\draw (0.6,0) edge [loop right, in = 25, out = 335,min distance=4mm] (0.6,0) ;
\end{tikzpicture}
\end{center}
Take the orientation $S$ of $X$ as follows:  Direct the loops arbitrarily, and for the non-loop edges, direct them from left to right.  Label the directed edges $S = \{s_{1},s_{2},s_{3},\ldots,s_{7} \}$ so that $s_{4},s_{5},s_{6}$ are the non-loop directed edges, $s_{1},s_{2},s_{3}$ the directed loops at the first vertex on the left, and $s_{7}$ the directed loop at the second vertex on the right.  Take the function $\alpha:S \rightarrow \mathbb{Z} \subseteq \mathbb{Z}_{3}$ given by
$$s_{i} \mapsto \alpha(s_{i}) = 1, s_{j} \mapsto \alpha(s_{j}) = 0, \text{ and } s_{7} \mapsto \alpha(s_{7}) = 11, $$
where $i=1,2,3$, and $j=4,5,6$.  Let us introduce some ramification at the second vertex only, say $I_{2} = 3\mathbb{Z}_{3}$.  We obtain the branched $\mathbb{Z}_{3}$-tower of graphs
\begin{equation*}
\begin{tikzpicture}[baseline={([yshift=-0.6ex] current bounding box.center)}]

%vertices
\draw[fill=black] (0,0) circle (0.7pt);
\draw[fill=black] (0.6,0) circle (0.7pt);        
                       
%edges
\draw (0,0) edge [bend left=20] (0.6,0);
\draw (0,0) edge (0.6,0);
\draw (0,0) edge [bend right=20] (0.6,0);
\draw (0,0) edge [loop left, in = 155, out = 205,min distance=4mm] (0,0) ;
\draw (0,0) edge [loop left, in = 145, out = 215,min distance=5mm] (0,0) ;
\draw (0,0) edge [loop left, in = 135, out = 225,min distance=6.5mm] (0,0) ;

\draw (0.6,0) edge [loop right, in = 25, out = 335,min distance=4mm] (0.6,0) ;
\end{tikzpicture}
\, \, \, \leftarrow \, \, \,
\begin{tikzpicture}[baseline={([yshift=-0.6ex] current bounding box.center)}]
% create the node
\node[draw=none,minimum size=2cm,regular polygon,regular polygon sides=3] (a) {};
\node[draw=none, minimum size=1cm,regular polygon,regular polygon sides=3] (b) {};

% draw a black dot in each vertex
\foreach \x in {1,2,...,3}
\fill (a.corner \x) circle[radius=0.7pt];

\foreach \y in {1,2,...,3}
\fill (b.corner \y) circle[radius=0.7pt];
  
\path (a.corner 1) edge (a.corner 2);
\path (a.corner 2) edge (a.corner 3);
\path (a.corner 3) edge (a.corner 1);

\path (a.corner 1) edge [bend left=10] (a.corner 2);
\path (a.corner 2) edge [bend left=10] (a.corner 3);
\path (a.corner 3) edge [bend left=10] (a.corner 1);

\path (a.corner 1) edge [bend right=10] (a.corner 2);
\path (a.corner 2) edge [bend right=10] (a.corner 3);
\path (a.corner 3) edge [bend right=10] (a.corner 1);

\path (a.corner 1) edge [bend left=15] (b.corner 1);
\path (a.corner 2) edge [bend left=15] (b.corner 2);
\path (a.corner 3) edge [bend left=15] (b.corner 3);

\path (a.corner 1) edge [bend right=15] (b.corner 1);
\path (a.corner 2) edge [bend right=15] (b.corner 2);
\path (a.corner 3) edge [bend right=15] (b.corner 3);

\path (a.corner 1) edge (b.corner 1);
\path (a.corner 2) edge (b.corner 2);
\path (a.corner 3) edge (b.corner 3);

\path (b.corner 1) edge (b.corner 3);
\path (b.corner 2) edge (b.corner 1);
\path (b.corner 3) edge (b.corner 2);

\end{tikzpicture}
\, \, \, \leftarrow \, \, \,
\begin{tikzpicture}[baseline={([yshift=-0.6ex] current bounding box.center)}]
% create the node
\node[draw=none,minimum size=2cm,regular polygon,regular polygon sides=9] (a) {};
\node[draw=none, minimum size=1cm,regular polygon,regular polygon sides=3] (b) {};
 
% draw a black dot in each vertex
\foreach \x in {1,2,...,9}
\fill (a.corner \x) circle[radius=0.7pt];
  
\foreach \y in {1,2,3}
\fill (b.corner \y) circle[radius=0.7pt];
  
\foreach \x\z in {1/2,2/3,3/4,4/5,5/6,6/7,7/8,8/9,9/1}
\path (a.corner \x) edge (a.corner \z);  
\foreach \x\z in {1/2,2/3,3/4,4/5,5/6,6/7,7/8,8/9,9/1}
\path (a.corner \x) edge [bend left=10] (a.corner \z);  
\foreach \x\z in {1/2,2/3,3/4,4/5,5/6,6/7,7/8,8/9,9/1}
\path (a.corner \x) edge [bend right=10] (a.corner \z); 
 
\foreach \x\z in {1/1,2/2,3/3,4/1,5/2,6/3,7/1,8/2,9/3} 
\path (a.corner \x) edge (b.corner \z);
\foreach \x\z in {1/1,2/2,3/3,4/1,5/2,6/3,7/1,8/2,9/3} 
\path (a.corner \x) edge [bend left=15] (b.corner \z);
\foreach \x\z in {1/1,2/2,3/3,4/1,5/2,6/3,7/1,8/2,9/3} 
\path (a.corner \x) edge [bend right=15] (b.corner \z);

\path (b.corner 1) edge [bend left=10] (b.corner 3);
\path (b.corner 1) edge [bend right=10] (b.corner 3);
\path (b.corner 1) edge  (b.corner 3);

\path (b.corner 2) edge [bend left=10] (b.corner 1);
\path (b.corner 2) edge [bend right=10] (b.corner 1);
\path (b.corner 2) edge  (b.corner 1);

\path (b.corner 3) edge [bend left=10] (b.corner 2);
\path (b.corner 3) edge [bend right=10] (b.corner 2);
\path (b.corner 2) edge  (b.corner 3);
  
\end{tikzpicture}
\, \, \, \leftarrow \, \, \,
\begin{tikzpicture}[baseline={([yshift=-0.6ex] current bounding box.center)}]
% create the node
\node[draw=none,minimum size=2cm,regular polygon,regular polygon sides=27] (a) {};
\node[draw=none, minimum size=1cm,regular polygon,regular polygon sides=3] (b) {};

% draw a black dot in each vertex
\foreach \x in {1,2,...,27}
\fill (a.corner \x) circle[radius=0.7pt];
  
\foreach \y in {1,2,3}
\fill (b.corner \y) circle[radius=0.7pt];
  
\foreach \x\z in {1/2,2/3,3/4,4/5,5/6,6/7,7/8,8/9,9/10,10/11,11/12,12/13,13/14,14/15,15/16,16/17,17/18,18/19,19/20,20/21,21/22,22/23,23/24,24/25,25/26,26/27,27/1}
\path (a.corner \x) edge (a.corner \z);  
\foreach \x\z in {1/2,2/3,3/4,4/5,5/6,6/7,7/8,8/9,9/10,10/11,11/12,12/13,13/14,14/15,15/16,16/17,17/18,18/19,19/20,20/21,21/22,22/23,23/24,24/25,25/26,26/27,27/1}
\path (a.corner \x) edge [bend left=30] (a.corner \z); 
\foreach \x\z in {1/2,2/3,3/4,4/5,5/6,6/7,7/8,8/9,9/10,10/11,11/12,12/13,13/14,14/15,15/16,16/17,17/18,18/19,19/20,20/21,21/22,22/23,23/24,24/25,25/26,26/27,27/1}
\path (a.corner \x) edge [bend right=30] (a.corner \z); 
 
\foreach \x\z in {1/1,2/2,3/3,4/1,5/2,6/3,7/1,8/2,9/3,10/1,11/2,12/3,13/1,14/2,15/3,16/1,17/2,18/3,19/1,20/2,21/3,22/1,23/2,24/3,25/1,26/2,27/3} 
\path (a.corner \x) edge [bend left=10] (b.corner \z);
\foreach \x\z in {1/1,2/2,3/3,4/1,5/2,6/3,7/1,8/2,9/3,10/1,11/2,12/3,13/1,14/2,15/3,16/1,17/2,18/3,19/1,20/2,21/3,22/1,23/2,24/3,25/1,26/2,27/3} 
\path (a.corner \x) edge [bend right=10] (b.corner \z);
\foreach \x\z in {1/1,2/2,3/3,4/1,5/2,6/3,7/1,8/2,9/3,10/1,11/2,12/3,13/1,14/2,15/3,16/1,17/2,18/3,19/1,20/2,21/3,22/1,23/2,24/3,25/1,26/2,27/3} 
\path (a.corner \x) edge (b.corner \z);

\path (b.corner 1) edge [bend left=4] (b.corner 3);
\path (b.corner 1) edge [bend left=8] (b.corner 3);
\path (b.corner 1) edge [bend left=12] (b.corner 3);
\path (b.corner 1) edge [bend left=16] (b.corner 3);

\path (b.corner 1) edge  (b.corner 3);

\path (b.corner 1) edge [bend right=4] (b.corner 3);
\path (b.corner 1) edge [bend right=8] (b.corner 3);
\path (b.corner 1) edge [bend right=12] (b.corner 3);
\path (b.corner 1) edge [bend right=16] (b.corner 3);

\path (b.corner 1) edge [bend left=4] (b.corner 2);
\path (b.corner 1) edge [bend left=8] (b.corner 2);
\path (b.corner 1) edge [bend left=12] (b.corner 2);
\path (b.corner 1) edge [bend left=16] (b.corner 2);

\path (b.corner 1) edge  (b.corner 2);

\path (b.corner 1) edge [bend right=4] (b.corner 2);
\path (b.corner 1) edge [bend right=8] (b.corner 2);
\path (b.corner 1) edge [bend right=12] (b.corner 2);
\path (b.corner 1) edge [bend right=16] (b.corner 2);

\path (b.corner 3) edge [bend left=4] (b.corner 2);
\path (b.corner 3) edge [bend left=8] (b.corner 2);
\path (b.corner 3) edge [bend left=12] (b.corner 2);
\path (b.corner 3) edge [bend left=16] (b.corner 2);

\path (b.corner 3) edge  (b.corner 2);

\path (b.corner 3) edge [bend right=4] (b.corner 2);
\path (b.corner 3) edge [bend right=8] (b.corner 2);
\path (b.corner 3) edge [bend right=12] (b.corner 2);
\path (b.corner 3) edge [bend right=16] (b.corner 2);
  
\end{tikzpicture}
\, \, \, \leftarrow \ldots
\end{equation*}
Using \cref{connect} and \cref{conn_unr}, one checks that all the finite graphs $X_{n}$ are connected.  The power series $f_{X,\mathcal{I},\alpha}(T)$ is given by
$$f_{X,\mathcal{I},\alpha}(T) =  3(3T + 3T^{2} - 2T^{3} + \ldots) \in \mathbb{Z}\llbracket T \rrbracket \subseteq \mathbb{Z}_{3}\llbracket T \rrbracket,$$
so we should have $\mu({\rm Pic}_{\Lambda}^{0}) = 1$ and $\lambda({\rm Pic}_{\Lambda}^{0}) = 3-1 = 2$.  We calculate
$$\kappa_{0} = 3, \kappa_{1} = 3^{4} \cdot 7^{2},\kappa_{2} = 3^{12} \cdot 7^{2} \cdot 19^{2}, \kappa_{3} = 3^{32} \cdot 7^{2} \cdot 19^{2} \cdot 5779^{2},$$
$$\kappa_{4} = 3^{88} \cdot 7^{2} \cdot 19^{2} \cdot 3079^{2} \cdot 5779^{2} \cdot 62650261^{2}, \ldots,$$
where we let again $\kappa_{n} = \kappa(X_{n})$.  We have
$${\rm ord}_{3}(\kappa(X_{n})) = 3^{n} +2n-1, $$
for all $n \ge 1$.

\end{enumerate}

\bibliographystyle{alpha}
\bibliography{references}
\end{document}